\documentclass[12pt]{amsart}

\usepackage{mathtools} 
\usepackage[svgnames]{xcolor}
\definecolor{color1}{RGB}{27,158,119}
\definecolor{color2}{RGB}{217,95,2}
\definecolor{color3}{RGB}{117,112,179}
\definecolor{color4}{RGB}{231,41,138}

\usepackage[unicode=true,pdfusetitle, bookmarks=true,bookmarksnumbered=false,bookmarksopen=false, breaklinks=false,pdfborder={0 0 1},backref=false,colorlinks=true]{hyperref}
\hypersetup{
  pdftitle={Traveling waves in periodic metric graphs via spatial dynamics},
  pdfauthor={Stefan Le Coz, Dmitry E. Pelinovsky and Guido Schneider}
  pdfsubject={},
  pdfkeywords={}, 
  colorlinks=true,
  linkcolor=DarkBlue  ,          %
  citecolor=DarkRed,        %
  filecolor=DarkMagenta,      %
  urlcolor=DarkGreen,           %
}

\usepackage[normalem]{ulem}

\usepackage{tikz}
\usepackage{graphicx}  
\usepackage{amsmath}  
\usepackage{amssymb,amsthm}
\usepackage{color}
\numberwithin{equation}{section}

\newcommand{\R}{{\mathbb R}}\newcommand{\N}{{\mathbb N}}
\newcommand{\Z}{{\mathbb Z}}\newcommand{\C}{{\mathbb C}}

\DeclareMathOperator{\arccosh}{arccosh}

\DeclareMathOperator{\sech}{sech}
\DeclarePairedDelimiter{\norm}{\lVert}{\rVert}

\newtheorem{theorem}{Theorem}
\newtheorem{lemma}{Lemma}[section]
\newtheorem{definition}{Definition}[section]
\newtheorem{proposition}{Proposition}[section]

\newtheorem{remark}{Remark}[section]
\newtheorem{example}{Example}[section]
\newtheorem{assumption}{Assumption}[section]

\title[Traveling waves in periodic metric graphs]{Traveling waves in periodic metric graphs\\ via spatial dynamics}

\author[S. Le Coz]{Stefan Le Coz}
\address{Institut de Mathématiques de Toulouse ; UMR5219, Université de Toulouse ; CNRS, UPS IMT, F-31062 Toulouse Cedex 9,  France}
\email{stefan.le-coz@math.univ-toulouse.fr}

\author[D. Pelinovsky]{Dmitry E. Pelinovsky}
\address{Department of Mathematics, McMaster University, Hamilton, Ontario, L8S 4K1, Canada}
\email{pelinod@mcmaster.ca}

\author[G. Schneider]{Guido Schneider}
\address{
Institut f\"{u}r Analysis, Dynamik und Modellierung,
Universit\"{a}t Stuttgart,
Pfaffenwaldring 57, D-70569 Stuttgart, Germany
}
\email{guido.schneider@mathematik.uni-stuttgart.de}
\thanks{The work of S. L. C. is 
  partially supported by ANR-11-LABX-0040-CIMI and the ANR project NQG ANR-23-CE40-0005}

\date{\today}

\begin{document}

\begin{abstract}
The purpose of this work is to introduce a concept of traveling waves in the setting of periodic metric graphs. It is known that the nonlinear Schr\"{o}dinger (NLS) equation on periodic metric graphs can be reduced asymptotically on long but finite time intervals to the homogeneous NLS equation, which admits traveling solitary wave solutions. In order to address persistence of such traveling waves beyond finite time intervals, we formulate the existence problem for traveling waves via spatial dynamics. There exist no spatially decaying (solitary) waves because of an infinite-dimensional center manifold in the spatial dynamics formulation. Existence of traveling modulating pulse solutions which are solitary waves with small oscillatory tails at very long distances from the pulse core is proven by using a local center-saddle manifold. We show that the variational formulation fails to capture existence of such modulating pulse solutions even in the singular limit of zero wave speeds where true (standing) solitary waves exist. Propagation of a traveling solitary wave and formation of a small oscillatory tail outside the pulse core is shown in numerical simulations of the NLS equation on the periodic graph. 
\end{abstract}

\maketitle


\section{Introduction}
\label{sec-1}

Traveling waves in nonlinear models have been widely studied due to their important applications in physics, chemistry, and biology. 

If the nonlinear model is formulated as a partial differential equation (PDE) for $u(t,x) : \R \times \R \to \R$ in time $t$ and spatial coordinate $x$, then traveling waves are solutions of the form $u(t,x) = \phi(x - ct)$, where $c$ is the wave speed and $\phi(x) : \R \to \R$ is the wave profile satisfying an ordinary differential equation (see e.g. \cite{Pava}). 

If the nonlinear model is formulated as a lattice differential equation (LDE) for $\{ u_j(t) \}_{j \in \Z} : \R \to \R^{\Z}$ in time $t$ on the lattice $\Z$, then traveling waves are solutions of the form $u_j(t) = \phi(hj-ct)$, where $h$ is the lattice spacing, $c$ is the wave speed, and $\phi : \R \to \R$ is the wave profile satisfying a differential advance-delay equation (see e.g. \cite{Pankov}). 

In the context of traveling modulating pulses in spatially homogeneous systems or standing modulating pulses in spatially periodic systems, solutions of nonlinear PDEs have been considered in the form 
$u(t,x) = \phi(x-ct,x)$, where $\phi$ is periodic with respect to $x$ 
and spatially decaying with respect to $\xi = x-ct$ with $c$ being the wave speed (see e.g.  \cite{Groves1,Groves2,LeScarret,PelSch08}). Methods of spatial dynamics have been used to construct modulating pulse solutions which decay to oscillatory tails of sufficiently small amplitude. Traveling modulating pulses in spatially periodic systems involve three spatial scales and solutions are written in the form $u(t,x) = \phi(\xi,z,x)$, where $\phi$ is periodic with respect to $x$ and $z =  kx - \sigma t$ and spatially decaying with respect to $\xi = x - ct$ with $k$ being the spatial wave number, $\sigma$ being the temporal wave number, and $c$ being the wave speed \cite{DPS}.  

The purpose of this work is to introduce a concept of traveling waves in a nonlinear PDE posed on a periodic metric graph. For the sake of simplicity, we consider the focusing cubic nonlinear Schr\"{o}dinger  equation (NLS). In the setting of metric graphs, the NLS equation has been widely studied due to many applications in quantum communications and optics, see the surveys \cite{DovettaRev,KNP22}.

\subsection{Formulation of the evolution model}

Let $\Gamma$ be a periodic graph written in the form 
$$
\Gamma = \bigoplus_{n \in \Z} \Gamma_n,
$$
where each cell $\Gamma_n$ is identical to others for $n \in \Z$. For the particular necklace graph shown in Figure \ref{fig:neckless}, $\Gamma_n$ includes a line segment $\Gamma_{n,0}$ of length $L_1$ and two semi-circles $\Gamma_{n,+} \oplus  \Gamma_{n,-}$ each of length $L_2$. 
If the period of $\Gamma$ is normalized to $2\pi$, then $L_1 + L_2 = 2\pi$ 
whereas $\Gamma_{n,0} $ and $\Gamma_{n,\pm} $ are identified isometrically with the intervals $I_{n,0} = [2 \pi n, 2 \pi n+ L_1 ]$ and 
$I_{n,\pm} = [2 \pi n+ L_1 ,2 \pi (n+1) ]$.

 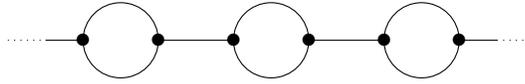
\begin{figure}[htpb!]
 	\centering
	\begin{tikzpicture}
	\node at (-2.5,0) {$\bullet$}; \node at (-1.5,0) {$\bullet$};
	\node at (-0.5,0) {$\bullet$}; \node at (0.5,0) {$\bullet$};
	\node at (1.5,0) {$\bullet$}; \node at (2.5,0) {$\bullet$};
	\draw (-2,0) circle (0.5); \draw (0,0) circle (0.5); \draw (2,0)
	circle (0.5); \draw (-1.5,0) -- (-0.5,0); \draw (0.5,0) --
	(1.5,0); \draw (-3,0) -- (-2.5,0); \draw (2.5,0) -- (3,0);
	\draw[dotted] (-3.5,0) -- (-3,0); \draw[dotted] (3,0) --
	(3.5,0);
	\end{tikzpicture}\hfill
	\caption{Schematic representation of a necklace graph.}
	\label{fig:neckless}
\end{figure}

Let $\psi(t,x) : \R \times \Gamma \to \mathbb{C}$ be the wave function defined  piecewise with components $\psi_{n,0}$ on $ I_{n,0}  $ and components $ \psi_{n,\pm} $ on $ I_{n,\pm}  $ subject to the following Neumann--Kirchhoff (NK) boundary conditions at each vertex point:
\begin{equation}
\begin{cases}    
\psi_{n,0}(t,2 \pi n+ L_1) = \psi_{n,+}(t,2 \pi n+ L_1) = \psi_{n,-}(t,2 \pi n+ L_1),  \\
\partial_x \psi_{n,0}(t,2 \pi n+ L_1) = \partial_x \psi_{n,+}(t,2 \pi n+ L_1 )+ \partial_x \psi_{n,-}(t,2 \pi n+ L_1),
\end{cases}
\label{vertex-1}
\end{equation}
and
\begin{equation}
\begin{cases}
\psi_{n,+}(t,2 \pi (n+1))= \psi_{n,-}(t,2 \pi (n+1)) = \psi_{n+1,0}(t,2 \pi (n+1)),\\
\partial_x \psi_{n,+}(t,2 \pi (n+1)) + \partial_x \psi_{n,-}(t,2 \pi (n+1)) = \partial_x \psi_{n+1,0}(t,2 \pi (n+1)). 
\end{cases}
\label{vertex-2}
\end{equation}
The focusing cubic NLS equation can be written in the normalized form:
\begin{equation} \label{NLS}
i \partial_t \psi +  \partial_x^2 \psi + 2 |\psi|^2 \psi = 0, \quad \psi(t,x) : \R \times \Gamma \to \mathbb{C}.
\end{equation}

Standing time-periodic and spatially decaying wave solutions of the NLS equation on metric graphs have been considered in many details in the recent literature, see \cite{KNP22} for a survey. For the necklace graph, the NLS equation (\ref{NLS}) on $\Gamma$ was reduced to homogeneous NLS and Dirac equations on the real line $\R$ by using homogenization methods in \cite{GilgPS}. This gives a finite-time approximation of the standing time-periodic and spatially decaying waves by the solitary wave solutions of the NLS and Dirac equations. Methods of discrete spatial dynamics were used in \cite{PS-2017} to obtain the standing wave solutions directly due to separation of the variables $t \in \mathbb{R}$ and $x \in \Gamma$. Furthermore, variational methods were developed to study the standing wave solutions as minimizers of energy functionals, see e.g. \cite{Dovetta,Pankov-per}. Standing waves in the limit of large energies were also studied by using asymptotic methods in \cite{Berkolaiko}. Extension of the existence results for standing time-periodic waves in the nonlinear Klein--Gordon equation on the necklace graph can be found in \cite{Maier}. 

Traveling long waves were considered by using the homogenization of the regularized Boussinesq equation on the periodic graph 
$\Gamma$ within the validity of the KdV (Korteweg--de Vries) approximation \cite{Dull}. However, the variables $t \in \mathbb{R}$ and $x \in \Gamma$ are not separated for traveling waves and hence it is not clear how to formulate the existence problem for such traveling wave solutions. This problem is addressed in our work. 

\subsection{Main results}

In order to introduce the main results, we recall the Floquet--Bloch spectral theory for the operator $-\partial_x^2$ on $L^2(\Gamma)$ (see \cite{GilgPS,PS-2017} and earlier works \cite{KL,KZ,Niikuni}). The spectrum of $-\partial_x^2$ in $L^2(\Gamma)$ is purely continuous and can be recovered with the Bloch eigenfunctions 
\begin{equation}
\label{Bloch-wave}
w(x) = e^{i \ell x} f(\ell,x), \quad f(\ell,x) = f(\ell,x+2\pi) = f(\ell+1,x) e^{ix}, \quad (\ell,x) \in \R \times \Gamma,
\end{equation}
where $f(\ell,x) = f(\ell,x+2\pi)$ can be written in the component form as 
$f_n(\ell,x) = f_{n+1}(\ell,x+2\pi)$ for $x \in \Gamma_n$.
The Bloch eigenfunctions are found from solutions of the spectral problem 
\begin{equation}
\label{spectral-prob}
-(\partial_x + i\ell)^2 f(\ell,x) = \omega(\ell) f(\ell,x), \quad (\ell,x) \in \R \times \Gamma_0
\end{equation}
subject to the boundary conditions
\begin{equation}
\begin{cases}
f_{0}(\ell,L_1) = f_{+}(\ell,L_1) = f_{-}(\ell,L_1),  \\
(\partial_x + i \ell) f_{0}(\ell,L_1) = (\partial_x + i \ell) f_{+}(\ell,L_1) + (\partial_x + i \ell) f_{-}(\ell,L_1),  
\end{cases}
\label{vertex-11}
\end{equation}
and
\begin{equation}
\begin{cases}
f_{+}(\ell,2 \pi)= f_{-}(\ell,2 \pi) = f_{0}(\ell,0),\\
(\partial_x + i \ell) f_{+}(\ell,2 \pi) + (\partial_x + i \ell) f_{-}(\ell,2 \pi) = (\partial_x + i \ell) f_{0}(\ell,0), 
\end{cases}
\label{vertex-12}
\end{equation}
where $\Gamma_0$ is the basic cell of the periodic graph $\Gamma$. 
By the spectral theorem, the spectrum of $-\partial_x^2$ in $L^2(\Gamma)$ is given by the union of the spectral bands 
\begin{equation}
\sigma(-\partial_x^2) = \bigcup_{m \in \N} \{ \omega_m(\ell): \;\; \ell \in \mathbb{B} \},
\label{specrum}
\end{equation}
where $\mathbb{B} := [-\frac{1}{2},\frac{1}{2})$ is the so-called Brillouin zone, to which $(\omega_m)$ might be restricted by periodicity. 

The homogenization result from \cite{GilgPS} can be stated as follows. Pick $m_0 \in \N$ and $\ell_0 \in \mathbb{B}$ and assume that $\omega_m(\ell_0) \neq \omega_{m_0}(\ell_0)$ for every $m \neq m_0$. For every $C_0 > 0$ and $T_0 > 0$, there exists $\varepsilon_0 > 0$ and $C > 0$ such that for all solutions $A \in C^0(\mathbb{R},H^3(\R))$ of the homogeneous NLS equation 
	\begin{equation}
	\label{NLS-homog}
	i \partial_T A + \frac{1}{2} \omega_{m_0}''(\ell_0) \partial_X^2 A + 2 \gamma |A|^2 A = 0, \quad \gamma := \frac{\| f_{m_0}(\ell_0,\cdot) \|^4_{L^4(\Gamma_0)}}{\| f_{m_0}(\ell_0,\cdot) \|^2_{L^2(\Gamma_0)}},
	\end{equation}
	satisfying 
	$$
	\sup_{T \in [0,T_0]} \|A(T,\cdot) \|_{H^3} \leq C_0
	$$
	and for all $\varepsilon \in (0,\varepsilon_0)$, there are solutions 
	$\psi \in C^0([0,T_0/\varepsilon^2],L^{\infty}(\Gamma))$ of the NLS equation (\ref{NLS}) satisfying 
	\begin{equation}
    \label{eq:thm1_Gilg}	    
	\sup_{t \in [0,T_0/\varepsilon^2]} \sup_{x \in \Gamma} | \psi(t,x) - \varepsilon A(\varepsilon^2 t, \varepsilon (x-ct)) f_{m_0}(\ell_0,x) e^{i \ell_0 x} e^{-i \omega_{m_0}(\ell_0)t} | \leq C \varepsilon^{3/2},
	\end{equation}
	where $c = \omega'_{m_0}(\ell_0)$.

The focusing cubic NLS equation (\ref{NLS-homog}) on the real line $\mathbb{R}$ admits a family of traveling ${\rm sech}$-soliton solutions if $\omega_{m_0}''(\ell_0) > 0$ since $\gamma > 0$. 
By \eqref{eq:thm1_Gilg}
if the initial data for $\psi(0,\cdot) \in L^{\infty}(\Gamma)$ is close to the family of traveling ${\rm sech}$-solitons, the solution of the NLS equation (\ref{NLS}) on the periodic graph $\Gamma$ is approximated by the traveling solitary waves on long but finite time intervals. The following two main questions are addressed in this work:
\begin{itemize}
	\item Can we define the traveling wave solutions on the periodic graph $\Gamma$ for all times? 
	\item Do the traveling waves remain spatially decaying for all times?
\end{itemize}

For the first question, if we think of the traveling wave reductions of the NLS equation (\ref{NLS}) on the periodic graph $\Gamma$, we can try substitutions of the form of either $\psi(t,x) = \psi(x-ct,x)$ or $\psi(t,x) = \psi(t,x-ct)$. However, these solution forms mix up the boundary conditions (\ref{vertex-1}) and (\ref{vertex-2}) which are no longer defined at the time-independent points of the spatial line. Therefore, these solution forms do not express the concept of traveling waves which move across the periodic graph $\Gamma$ with a constant propagation speed. 

As the main results of our paper, we give a positive answer to the first question and a generally negative answer to the second question. We define traveling waves by using the spatial dynamics formulation of the NLS equation 
(\ref{NLS}) on the periodic graph $\Gamma$. Then, we study the linearized system near the zero equilibrium and show that there exist infinitely many purely imaginary eigenvalues of the linearized system which correspond to the center manifold of the spatial dynamical system. As a result, no true 
spatially decaying solutions to zero (homoclinic orbits) exist generally at the intersection of stable and unstable manifolds of the spatial dynamical system. The best approximation of the spatially decaying traveling wave solutions are traveling modulating pulses with small oscillatory tails far away from the pulse center, similarly to the traveling modulating pulses in a spatially periodic PDE studied in \cite{DPS} and similarly to the homoclinic orbits in a temporally periodic LDE studied in \cite{Chong}. Such solutions are identified in the following theorem.

\begin{theorem}
	\label{theorem-NLS-periodic}
	Pick $m_0 \in \N$ and $\ell_0 \in \mathbb{B}$ with $\omega_{m_0}''(\ell_0) > 0$ and define 
	$$
	\begin{cases}
	\sigma_0 = -\omega_{m_0}(\ell_0) + \ell_0 \omega_{m_0}'(\ell_0), \\
	c_0 = \omega_{m_0}'(\ell_0).
	\end{cases}
	$$
	Assume that the set of purely imaginary roots $\lambda \in i \R$  of the characteristic equations 
	$$
	\sigma_0 + i c_0 \lambda + \omega_m(-i \lambda) = 0, \quad m \in \N,
	$$
	admit only simple nonzero roots, with the exception of a double zero root for $m = m_0$ and $\ell = \ell_0$. For every $N_0 > 0$, there exist $\varepsilon_0 > 0$ and $C_0 > 0$ such that for every $\varepsilon \in (0,\varepsilon_0)$, the NLS equation (\ref{NLS}) admits a traveling wave solution of the form
	$$
	\psi_n(t,x) = \phi(x - ct,x - 2\pi n) e^{i \sigma t}, \qquad \phi(\xi,x) : \R \times \Gamma_0 \to \C
	$$
	with $\sigma = \sigma_0 + \varepsilon^2$, $c = c_0$ and $\phi$ satisfying 
	$$
\sup_{\xi \in [-N_0 \varepsilon^{-1},N_0 \varepsilon^{-1}]} \sup_{x \in \Gamma_0} | \phi(\xi,x) - \varepsilon A(\varepsilon \xi) f_{m_0}(\ell_0,x) e^{i \ell_0 \xi} | \leq C_0 \varepsilon^{2},
$$	
where $A \in H^{\infty}(\R)$ is a solution of the stationary homogeneous NLS equation 
	\begin{equation}
	\label{NLS-stat}
\frac{1}{2} \omega_{m_0}''(\ell_0) A'' - A + 2 \gamma |A|^2 A = 0, \quad \gamma := \frac{\| f_{m_0}(\ell_0,\cdot) \|^4_{L^4(\Gamma_0)}}{\| f_{m_0}(\ell_0,\cdot) \|^2_{L^2(\Gamma_0)}}.
	\end{equation}
\end{theorem}

We note that the stationary NLS equation (\ref{NLS-stat}) is a traveling wave reduction of the time-dependent NLS equation (\ref{NLS-homog}), hence the {\rm sech}-solitons of the NLS equation (\ref{NLS-homog}) give the leading-order approximation of the traveling wave solutions of the NLS equation (\ref{NLS}) on the periodic graph $\Gamma$. These solutions do not decay to zero for $|\xi| \geq N_0 \varepsilon^{-1}$ but approach the oscillatory remainder term.

The main novelty and significance of this work is not the proof of Theorem \ref{theorem-NLS-periodic} but rather the definition and analysis of the existence problem for traveling waves on periodic metric graphs. 
If the periodic graph $\Gamma$ is replaced by a ladder graph shown in Figure \ref{fig:ladder}, then the spatial dynamical system has exact solutions for traveling solitary waves, due to the additional symmetry of the ladder graph. This gives one example where the oscillatory tails far away from the pulse center are identically zero in Theorem \ref{theorem-NLS-periodic}.

  \begin{figure}[htpb!]
  	\centering
	\begin{tikzpicture}[rotate=90]
	\def\ladderHeight{6} 
	\def\spacing{1.5} 
	\def\railDistance{1.5} 
	\def\dotSize{0.1} 
	
	\foreach \x in {0, \railDistance} {
		\draw (\x, -0.5*\spacing) -- (\x, \ladderHeight * \spacing+0.5*\spacing);
		\draw[dotted] (\x, -0.5*\spacing) -- (\x,  -0.8*\spacing);
		\draw[dotted] (\x, \ladderHeight * \spacing+0.5*\spacing) -- (\x, \ladderHeight * \spacing+0.8*\spacing);
	}
	
	\foreach \y in {0,...,\ladderHeight} {
		\draw (0, \y * \spacing) -- (\railDistance, \y * \spacing);
		
		\fill (0, \y * \spacing) circle (\dotSize);
		\fill (\railDistance, \y * \spacing) circle (\dotSize);
	}
	\end{tikzpicture}\hfill
	\caption{Schematic representation of a ladder graph.}
	\label{fig:ladder}
\end{figure}
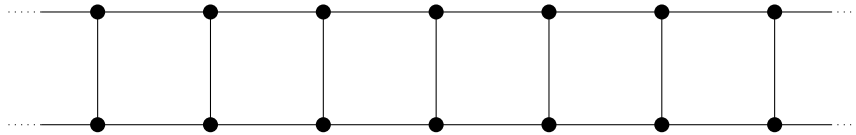

\subsection{Organization of the paper}

In Section \ref{sec-2}, we give the spatial dynamics formulation of the existence problem for traveling waves of the NLS equation (\ref{NLS}) on the necklace graph $\Gamma$ of Figure \ref{fig:neckless}. In the homogeneous case with $L_2 = 0$, this formulation admits exact traveling {\rm sech}-soliton solutions, as we show in Appendix \ref{app-A}. In the general case, the spatial dynamical system can be formulated variationally but, as we show in Appendix \ref{app-B}, no minimizers of the energy functional exist even in the singular limit of zero wave speed where the true standing solitary waves exist. 

Section \ref{sec-3} contains the analysis of the linearized system at the zero solution. We show that the linearized system admits infinite-dimensional center, stable, and unstable manifolds. Bifurcations of small traveling waves are identified from the existence and splitting of double eigenvalues off the imaginary axis.  However, the presence of infinitely many purely imaginary eigenvalues suggests that the traveling modulating pulse solutions generally exhibit small oscillatory tails at very long distances from the pulse core. 

Bifurcations of traveling modulating pulses are studied in Section \ref{sec-7}, where we formally derive the stationary NLS equation (\ref{NLS-stat}) and construct a local center-stable manifold of the spatial dynamical system satisfying the reversibility constraints. Estimates on the center-stable manifold justify the approximation bound formulated in Theorem \ref{theorem-NLS-periodic}. 

Section \ref{sec-numerics} illustrates the robustness of traveling solitary waves with the help of numerical simulations of the initial-value problem for the NLS equation (\ref{NLS}). We do observe the small oscillatory tails outside the pulse core but their amplitude is small if the parameters of the traveling waves are close to the bifurcation threshold. 

Section \ref{app-C} gives the analysis of the traveling waves for the ladder graph in Figure \ref{fig:ladder}, where the spatial dynamical system admits exact solutions for true traveling solitary waves. We show that the spatial dynamical system also admits other bifurcations but no true solitary waves generally exist for other bifurcations due to the same reason as for the necklace graph in Figure \ref{fig:neckless}.

\section{Spatial dynamics formulation}
\label{sec-2}

To define the traveling wave solutions for the NLS equation 
(\ref{NLS}), we first reset the wave function as an infinite-dimensional vector on a fixed spatial domain. In other words, we replace
$\psi_n(t,x-2\pi n)$ for $x \in \Gamma_n$ and $n \in \Z$ by 
$\psi_n(t,x)$, where the new component $\psi_n(t,x)$ is now defined for $(t,x) \in \R\times \Gamma_0$, with $\Gamma_0$ being the basic cell of the periodic graph $\Gamma$. Traveling waves are solutions of the form:
\begin{equation}
\psi_n(t,x) = \psi_{n+1}(t + 2\pi c^{-1},x) e^{-i 2\pi c^{-1} \sigma}, \quad t \in \R, \quad x \in \Gamma_0,
\label{trave-solutions-components}
\end{equation}
with the two parameters given by the propagation speed $c \neq 0$ and the frequency $\sigma \in \mathbb{R}$. The reason why the solution form (\ref{trave-solutions-components}) expresses the concept of traveling waves is that the same profile is repeated at each site $n \in \Z$ of the lattice with some time delay $2\pi c^{-1}$ between the two adjacent sites and with rotation of the complex phase. Solutions of the form (\ref{trave-solutions-components}) are equivalent to the traveling wave reduction:
\begin{equation}
\label{trav-solution}
\psi_n(t,x) = \phi(2\pi n - ct, x) e^{i \sigma t}, \quad \phi(\xi,x) : \R \times \Gamma_0 \to \mathbb{C}.
\end{equation}
The NLS equation (\ref{NLS}) for the traveling wave solutions of the form (\ref{trav-solution}) can be rewritten in the equivalent form 
\begin{equation} \label{NLS-spatial}
-i c \partial_\xi \phi +  \partial_x^2 \phi + 2 |\phi|^2 \phi = \sigma \phi, \quad \phi(\xi,x) : \R \times \Gamma_0 \to \mathbb{C},
\end{equation}
where $\phi(\xi,\cdot)$ consists of the component $\phi_0(\xi,\cdot)$ defined on $[0,L_1]$ and 
the components $\phi_{\pm}(\xi,\cdot)$ defined on $[L_1,2\pi]$ subject to the boundary conditions:
\begin{equation}
\begin{cases} \phi_{0}(\xi,L_1) = \phi_{+}(\xi,L_1) = \phi_{-}(\xi,L_1),  \\
\partial_x \phi_{0}(\xi,L_1) = \partial_x \phi_{+}(\xi,L_1) + \partial_x \phi_{-}(\xi,L_1),  \end{cases} \label{vertex-3}
\end{equation}
and 
\begin{equation}
\begin{cases} \phi_{+}(\xi,2\pi) = \phi_{-}(\xi,2\pi) = \phi_{0}(\xi+2\pi,0),  \\
\partial_x \phi_{+}(\xi,2\pi) + \partial_x \phi_{-}(\xi,2\pi) = \partial_x \phi_{0}(\xi+2\pi,0).  \end{cases} \label{vertex-4}
\end{equation}
Boundary conditions (\ref{vertex-3})--(\ref{vertex-4}) follows from 
(\ref{vertex-1})--(\ref{vertex-2}) for the solutions of the form (\ref{trav-solution}).

The system of equations (\ref{NLS-spatial}) is nonlocal because the boundary conditions (\ref{vertex-4}) at the end points $x = 0$ and $x = 2\pi$ are related with the $2\pi$ shift with respect to the variable $\xi \in \R$. In particular, the operator $\partial^2_x$ in $L^2(\Gamma_0)$ equipped with the boundary conditions (\ref{vertex-3}) and (\ref{vertex-4}) cannot be defined independently of $\xi \in \mathbb{R}$.

In order to obtain a local formulation of the spatial dynamical system, we introduce the substitution 
\begin{equation}
\label{phi-tilde-phi}
\phi(\xi,x) = \tilde{\phi}(\xi+x,x), \quad \tilde{\phi}(\xi,x) : \R\times \Gamma_0 \to \C.
\end{equation}
The NLS equation (\ref{NLS-spatial}) with the substitution (\ref{phi-tilde-phi}) is rewritten in the equivalent form:
\begin{equation}
\label{NLS-potential}
- i c \partial_{\xi} \tilde{\phi} +  (\partial_{\xi} + \partial_x)^2 \tilde{\phi} + 2 |\tilde{\phi}|^2 \tilde{\phi}  = \sigma \tilde{\phi},
\end{equation}
which yields the spatial dynamical system in the form:
\begin{equation}
\label{NLS-dynamics}
\frac{d}{d\xi} \begin{pmatrix} \tilde{\phi} \\ \tilde{\eta} \end{pmatrix} = 
\begin{pmatrix} \tilde{\eta} \\
-\partial_x^2 \tilde{\phi} - 2 \partial_x \tilde{\eta} 
+ i c \tilde{\eta} + \sigma \tilde{\phi} - 2 |\tilde{\phi}|^2 \tilde{\phi} \end{pmatrix},
\end{equation}
where $\tilde{\eta}(\xi,x) := \partial_{\xi} \tilde{\phi}(\xi,x)$. The boundary conditions (\ref{vertex-3}) and (\ref{vertex-4}) with the substitution (\ref{phi-tilde-phi}) are rewritten in the following form:
\begin{equation}
\begin{cases} \tilde{\phi}_{0}(\xi,L_1) = \tilde{\phi}_{+}(\xi,L_1) = \tilde{\phi}_{-}(\xi,L_1),  \\
\tilde{\eta}_{0}(\xi,L_1) = \tilde{\eta}_{+}(\xi,L_1) = \tilde{\eta}_{-}(\xi,L_1),  \\ 
    	\partial_x \tilde{\phi}_{0}(\xi,L_1) + \tilde{\eta}_{0}(\xi,L_1) =\partial_x \tilde{\phi}_{+}(\xi,L_1) + \tilde{\eta}_{+}(\xi,L_1) + \partial_x \tilde{\phi}_{-}(\xi,L_1) + \tilde{\eta}_{-}(\xi,L_1), 
        \end{cases} \label{vertex-5}
	\end{equation}
and
\begin{equation}
\begin{cases} \tilde{\phi}_{+}(\xi,2\pi) = \tilde{\phi}_{-}(\xi,2\pi) = \tilde{\phi}_{0}(\xi,0),  \\
\tilde{\eta}_{+}(\xi,2\pi) = \tilde{\eta}_{-}(\xi,2\pi) = \tilde{\eta}_{0}(\xi,0),  \\
\partial_x \tilde{\phi}_{+}(\xi,2\pi) + \tilde{\eta}_{+}(\xi,2\pi) + 
\partial_x \tilde{\phi}_{-}(\xi,2\pi) + \tilde{\eta}_{-}(\xi,2\pi) = 
\partial_x \tilde{\phi}_{0}(\xi,0) + \tilde{\eta}_{0}(\xi,0),  \end{cases} \label{vertex-6}
\end{equation}
where we have translated $\xi$ by $2\pi$ in (\ref{vertex-6}). Compared to (\ref{vertex-4}), 
the boundary conditions (\ref{vertex-6}) are defined independently of $\xi \in \mathbb{R}$.

The operators $\partial_x^2$ and $i \partial_x$ in (\ref{NLS-dynamics}) act on $\tilde{\phi}$ and $\tilde{\eta}$ respectively, which satisfy the boundary  conditions (\ref{vertex-5}) and (\ref{vertex-6}). We define 
$$
H^1_{\rm C} := \left\{ (\tilde{\eta}_0,\tilde{\eta}_+,\tilde{\eta}_-) \in 
H^1([0,L_1]) \times H^1([L_1,2\pi]) \times H^1([L_1,2\pi]) \right\} \subset L^2(\Gamma_0)
$$
subject to the boundary conditions 
\begin{equation}
\label{bc-eta}
\tilde{\eta}_0(L_1) = \tilde{\eta}_+(L_1) = \tilde{\eta}_-(L_1), \qquad 
\tilde{\eta}_+(2\pi) = \tilde{\eta}_-(2\pi) = \tilde{\eta}_0(0).
\end{equation}
We define $H^2_{\rm NK} $ as
$$
H^2_{\rm NK} := \left\{ (\tilde{\phi}_0,\tilde{\phi}_+,\tilde{\phi}_-) \in 
H^2([0,L_1]) \times H^2([L_1,2\pi]) \times H^2([L_1,2\pi]) \right\} \subset H^1_C
$$ 
subject to the continuity conditions of $H^1_{\rm C}$ at the vertex points for $(\tilde{\phi}_0,\tilde{\phi}_+,\tilde{\phi}_-)$ and the following two conditions:
\begin{align}
\label{bc-phi-1}
\partial_x \tilde{\phi}_{0}(L_1) - \tilde{\eta}_0(L_1) &= \partial_x \tilde{\phi}_{+}(L_1) + \partial_x \tilde{\phi}_{-}(L_1), \\
\label{bc-phi-2}
\partial_x \tilde{\phi}_{+}(2\pi) + 
\partial_x \tilde{\phi}_{-}(2\pi) &= 
\partial_x \tilde{\phi}_{0}(0) - \tilde{\eta}_{0}(0),
\end{align}
where $(\tilde{\eta}_0,\tilde{\eta}_+,\tilde{\eta}_-) \in H^1_{\rm C}$.
The subspaces $H^2_{\rm NK}$, $ H^1_{\rm C}$ 
of $L^2(\Gamma_0)$ are suitable domains for the self-adjoint operators $(\partial_x^2,i\partial_x)$ acting on $(\tilde{\phi},\tilde{\eta})$ in the spatial dynamics formulation (\ref{NLS-dynamics}). Hence, the phase space for the spatial dynamical system (\ref{NLS-dynamics}) is given by 
\begin{equation}
\label{phase-space}
\mathcal{D} := \left\{ (\tilde{\phi},\tilde{\eta}) : \;\; \tilde{\phi} \in H^2_{\rm NK}, \;\; \tilde{\eta} \in H^1_{\rm C} \right\}.
\end{equation}
The mapping $\R \ni \xi \mapsto (\tilde{\phi},\tilde{\eta})(\xi,\cdot) \in \mathcal{D}$ for the solutions of the system (\ref{NLS-dynamics}) defines a curve in the phase space $\mathcal{D}$. The traveling solitary wave with the profile $\tilde{\phi}(\xi,x)$ corresponds to a homoclinic orbit connecting the trivial  equilibrium state $(0,0)$ as $\xi \to \pm \infty$. However, such traveling solitary wave solutions do not generally exist in the spatial dynamical system (\ref{NLS-dynamics}), as we show in Section \ref{sec-3} from the analysis of the linearization at $(0,0)$.

We conclude this section with two remarks. First, if $L_1 = 2\pi$ and $L_2 = 0$, the vertex conditions (\ref{vertex-5}) and (\ref{vertex-6}) are equivalent to the periodicity conditions so that $H^2_{\rm NK} = H^2_{\rm per}$ and $H^1_{\rm C} = H^1_{\rm per}$. As we show in Appendix \ref{app-A}, there exists an exact traveling ${\rm sech}$-soliton solution to the spatial dynamical system (\ref{NLS-dynamics}) in this periodic case. 

Second, the spatial NLS equation (\ref{NLS-potential}) can be formulated variationally as the Euler--Lagrange equation of the following action functional
\begin{equation}
\label{var-form}
\Lambda_{\sigma,c}(\varphi) := \int_{\mathbb{R} \times \Gamma_0} \left(  |(\partial_{\xi} + \partial_x) \varphi|^2 - |\varphi|^4 + \sigma |\varphi|^2 + \frac{i}{2} c (\bar{\varphi} \partial_{\xi} \varphi - \varphi \partial_{\xi} \bar{\varphi}) \right) d\xi dx. 
\end{equation}
However, as we show in Appendix \ref{app-B}, 
this variational formulation is not useful even for $c  = 0$ because the infimum of the quadratic part of $\Lambda_{\sigma,0}$ is not achieved under the constraint of fixed $\int_{\mathbb{R} \times \Gamma_0} |\varphi|^4 d\xi dx$.

Standing wave solutions of the NLS equation (\ref{NLS}) on the periodic graph $\Gamma$ were obtained in \cite{Dovetta} from minimization of the energy $E$ subject to fixed mass $Q$, 
where 
\begin{equation}
\label{standard-var}
E(\varphi) := \int_{\Gamma} \left( |\partial_x \varphi|^2 - |\varphi|^4 \right) dx, \qquad Q(\varphi) := \int_{\Gamma} |\varphi|^2 dx.
\end{equation}
Alternatively, the standing wave solutions can be obtained from minimization of $E(\varphi) + \sigma Q(\varphi)$
on the Nehari manifold (see \cite{Pankov-per}). Unfortunately, 
the variational techniques used to obtain the existence of standing waves with the functionals in (\ref{standard-var}) do not extend to the functional $\Lambda_{\sigma,c}(\varphi)$ given by (\ref{var-form}). 

The reason for the failure of the variational formulation (\ref{var-form}) is clear from the solution form (\ref{trave-solutions-components}) since the limit $c \to 0$ appears to be a singular limit for the traveling wave solutions of the form (\ref{trav-solution}). If $c \neq 0$, we are looking for a smooth solution $(\tilde{\phi},\tilde{\eta}) \in C^1(\mathbb{R},\mathcal{D})$ of the spatial dynamical system (\ref{NLS-dynamics}). However, the solution in the form (\ref{trav-solution}) with $c = 0$ needs not to be smooth and can be expressed in the separable form as $\psi_n(t,x) = \phi_n e^{i \sigma t}$ with $\phi \in \mathbb{C}^{\mathbb{Z}}$ obtained from the discrete map in \cite{PS-2017}. This situation is very similar to the LDEs, where the standing wave solutions satisfy the discrete maps and need not to be smooth, whereas the traveling wave solutions satisfy the differential advance-delay equations and need to be smooth. The limit of vanishing wave speed appears also as the singular limit of the differential advance-delay equations, see \cite{MalletParet} and \cite{P1,P2}.

\section{Linear theory}
\label{sec-3}

The following proposition summarizes the state-of-art on the spectral theory for the spectral problem (\ref{spectral-prob}) with boundary conditions (\ref{vertex-11}) and (\ref{vertex-12}).

\begin{proposition}
	\label{prop-spectrum}
There exists a sequence of eigenvalues $\{ \omega_m(\ell)\}_{m \in \N}$ 
in the spectral problem (\ref{spectral-prob}) with the boundary conditions (\ref{vertex-11}) and (\ref{vertex-12}) such that 
$$
\omega_m(\ell) = \omega_m(\ell+1), \qquad \ell \in \R, \quad m \in \N.
$$ 	
The sequence consists of two subsequencies. The first one at $\{\omega_m^{(1)}(\ell)\}_{m \in \N}$ with 
\begin{equation}
\label{omega-first}
\omega_m^{(1)} := \frac{\pi^2 m^2}{L_2^2}
\end{equation}
corresponds to the flat bands which are independent of $\ell \in \R$. The second one at $\{ \omega_m^{(2)}(\ell)\}_{m \in \N}$ is defined by roots of the characteristic equation:
\begin{equation}
\label{char-eq}
9 \cos(2\pi \sqrt{\omega}) - \cos(2 (\pi - L_2)\sqrt{\omega}) 
- 8\cos(2\pi \ell) = 0,
\end{equation}
where we have used $L_1 + L_2 = 2\pi$.
\end{proposition}

\begin{proof}
It is more efficient computationally (see, e.g., \cite{PS-2017}) to obtain the two subsequencies of  eigenvalues $\{ \omega_m(\ell)\}_{m \in \N}$ in the original variable $w(x)$ satisfying the system of differential equations 
\begin{equation*}
w''(x) + \omega w(x) = 0, \qquad x \in \Gamma.
\end{equation*}

For the first subsequence of eigenvalues, we have the reduction
\begin{equation}
\label{reduction-1}
\begin{cases}
w_{n,0}(x) = 0, \quad & x \in [2\pi n, 2\pi n + L_1], \\
w_{n,+}(x) = -w_{n,-}(x), \quad & x \in [2\pi n + L_1, 2\pi (n+1)],
\end{cases}
\quad \quad n \in \mathbb{Z}.
\end{equation}
The solution of $w''(x) + \omega w(x) = 0$ satisfying the boundary conditions (\ref{vertex-1}) and (\ref{vertex-2}) is given by 
$$
w_{n,\pm}(x) = \pm a_n \sin \sqrt{\omega} (x-2\pi n - L_1)
$$ 
with arbitrary $\{a_n\}_{n \in \Z}$ and uniquely defined $\omega$ in the set $\{ \frac{\pi^2 m^2}{L_2^2}\}_{m \in \mathbb{N}}$. This yields the first subsequence of eigenvalues (\ref{omega-first}) for the flat ($\ell$-independent) bands. 

For the second subsequence of eigenvalues, we have the reduction
\begin{equation}
\label{reduction-2}
w_{n,+}(x) = w_{n,-}(x), \quad x \in [2\pi n + L_1, 2\pi (n+1)],
\quad \quad n \in \mathbb{Z}
\end{equation}
with the corresponding solution 
\begin{equation*}
\begin{cases}
w_{n,0}(x) = a_n \cos \sqrt{\omega}(x-2\pi n) + b_n \sin \sqrt{\omega}(x-2\pi n), & x \in [2\pi n, 2\pi n + L_1], \\
w_{n,+}(x) = c_n \cos \sqrt{\omega}(x-2\pi n -L_1) + d_n \sin \sqrt{\omega}(x-2\pi n - L_1), & x \in [2\pi n + L_1, 2\pi (n+1)].
\end{cases}
\end{equation*}
The boundary conditions define the monodromy matrix 
$$
\begin{pmatrix} a_{n+1} \\ b_{n+1} \end{pmatrix} 
= M(\omega) \begin{pmatrix} a_{n} \\ b_{n} \end{pmatrix}, 
$$
where
\begin{equation}
\label{matrix-M}
M(\omega) = \begin{pmatrix} \cos \sqrt{\omega} L_2 & \sin \sqrt{\omega} L_2 \\ -2 \sin \sqrt{\omega} L_2 & 2 \cos \sqrt{\omega} L_2 \end{pmatrix}
\begin{pmatrix} \cos \sqrt{\omega} L_1 & \sin \sqrt{\omega} L_1 \\ -0.5 \sin \sqrt{\omega} L_1 & 0.5 \cos \sqrt{\omega} L_1 \end{pmatrix}.
\end{equation}
Since $\det M(\omega) = 1$, eigenvalues $\mu$ of $M(\omega)$ (called Floquet multipliers) are found from roots of the quadratic equations:
$$
\mu^2 - \mu \; {\rm tr} M(\omega) + 1 = 0, 
$$ 
where
$$
{\rm tr} M(\omega) = 2 \cos \sqrt{\omega} L_2 \cos \sqrt{\omega} L_1 - \frac{5}{2} \sin \sqrt{\omega} L_2 \sin \sqrt{\omega} L_1.
$$
By using $\mu = e^{2\pi i \ell}$, where $\ell \in \R$ is the Bloch wave number in the wave function representation $w(x) = e^{i \ell x} f(\ell,x)$,  we find the second subsequence of eigenvalues from roots of the characteristic equation (\ref{char-eq}).
\end{proof}

\begin{example}
	\label{example-spectrum}
	Two particular cases of the spectral theory in Proposition \ref{prop-spectrum} are given by 
\begin{itemize}
	\item In the symmetric case $L_1 = L_2 = \pi$, we obtain from (\ref{char-eq}):
\begin{equation}
\label{char-eq-sym}
9 \cos(2 \pi \sqrt{\omega}) - 1 - 8 \cos(2\pi \ell) = 0.
\end{equation}
\item In the homogeneous case $L_1 = 2\pi$, $L_2 = 0$, we obtain from (\ref{char-eq}):
\begin{equation}
\label{char-eq-hom}
 \cos(2\pi \sqrt{\omega}) - \cos(2\pi \ell) = 0.
\end{equation}
The sequence of eigenvalues $\{ \omega_m^{(2)}(\ell) \}_{m \in \N}$ can be equivalently represented as $\{ \omega_k^{(2)}(\ell) \}_{k \in \Z}$ with 
\begin{equation}
\label{seq-eig}
\omega_k^{(2)}(\ell) := (k+\ell)^2, \quad 
\ell \in \mathbb{B}, \quad k \in \Z,
\end{equation}
where $\mathbb{B} = [-\frac{1}{2},\frac{1}{2})$. At the same time, $\omega_m^{(1)} \to \infty$ as $L_2 \to 0$, hence the sequence of flat bands $\{ \omega_m^{(1)}(\ell) \}_{m \in \N}$ is not relevant for the homogeneous case. 
\end{itemize}
\end{example}

Adopting the spectral theory for the operator $-\partial_x^2$ on $L^2(\Gamma)$, 
we can obtain solutions of the linearization of the spatial dynamical system (\ref{NLS-dynamics}) at the trivial equilibrium  $(0,0)$. A perturbation $(\varphi,\zeta)$ of $(0,0)$ in variables $(\tilde{\phi},\tilde{\eta})$ satisfies the linearized system
\begin{equation}
\label{NLS-linear}
\frac{d}{d\xi} \begin{pmatrix} \varphi \\ \zeta \end{pmatrix} = 
\begin{pmatrix} \zeta \\
-\partial_x^2 \varphi - 2 \partial_x \zeta 
+ i c \zeta + \sigma \varphi \end{pmatrix}.
\end{equation}
The following lemma gives eigenvalues of the linearized system (\ref{NLS-linear}). 

\begin{lemma}
	\label{lem-linearized}
Solutions of the linearized system (\ref{NLS-linear}) can be written in the separated form
	$$
	\varphi(\xi,x) = e^{\lambda \xi} \hat{\varphi}(x), \quad 
	\zeta(\xi,x) = e^{\lambda \xi} \hat{\zeta}(x)
	$$
with $(\hat{\varphi},\hat{\zeta}) \in \mathcal{D}$, where $\mathcal{D}$ is given by (\ref{phase-space}) and $\lambda$ are roots of the characteristic equations
\begin{equation}
\label{char-eq-omega}
\sigma + i c \lambda + \omega_m(-i \lambda) = 0, \quad m \in \mathbb{N}, 
\end{equation}
with $\omega_m(\ell)$, $m \in \N$ being extended in $\ell \in \mathbb{C}$ from roots of the characteristic equations (\ref{char-eq}).
\end{lemma}
\begin{proof}
Using $\hat{\zeta} = \lambda \hat{\varphi}$ from the first equation in system (\ref{NLS-linear}), we rewrite the second equation in system (\ref{NLS-linear}) as the spectral problem:
\begin{equation}
\label{NLS-spectral} 
\sigma \hat{\varphi}  + i c \lambda \hat{\varphi} - (\partial_x + \lambda)^2 \hat{\varphi} = 0.
\end{equation}
By comparing (\ref{vertex-5})--(\ref{vertex-6}) with (\ref{vertex-11})--(\ref{vertex-12}) and using the Bloch eigenfunctions 
$f(\ell,x)$ with properties in (\ref{Bloch-wave}) and (\ref{spectral-prob}), we find the correspondence:
\begin{equation}
\label{eigenfunction}
\lambda = i \ell, \qquad \hat{\varphi} = f_m(-i\lambda,x), 
\end{equation}
which reduces the spectral problem (\ref{NLS-spectral}) to the characteristic equations (\ref{char-eq-omega}).
\end{proof}

\subsection{Bifurcation of double imaginary eigenvalues}

Among eigenvalues $\lambda$ of Lemma \ref{prop-spectrum}, we will be particularly interested in the double eigenvalues on $i \R$, which split away from $i \mathbb{R}$, when a point $(\sigma,c)$ deviates from the bifurcation point $(\sigma_0,c_0)$. 

\begin{definition}
We say that $(\sigma_0,c_0)$ is the bifurcation point if one of the characteristic equations (\ref{char-eq-omega}) admits a double eigenvalue $\lambda_0 = i \ell_0$ for some $\ell_0 \in \mathbb{B} = [-\frac{1}{2},\frac{1}{2})$ and $m_0 \in \mathbb{N}$. 
\end{definition}

The following lemma relates $(\sigma_0,c_0)$ to $\omega_{m_0}(\ell_0)$ and its derivative and gives the criterion that 
the double eigenvalue splits transversely to $i\R$ when $\sigma \neq \sigma_0$.

\begin{lemma}
	\label{prop-expansion}
	Let $(\sigma_0,c_0)$ be the bifurcation point for some $\ell_0 \in \mathbb{B}$ and $m_0 \in \mathbb{N}$. Then, we have
	\begin{align}
	\label{bif-point}
	\begin{cases} 
	\sigma_0 = -\omega_{m_0}(\ell_0) + \ell_0 \omega_{m_0}'(\ell_0), \\
	c_0 = \omega_{m_0}'(\ell_0). \end{cases}
	\end{align}
	Assume that $\omega_{m_0}''(\ell_0) \neq 0$. Then, there are $\varepsilon > 0$ 
	and $C > 0$ such that there exist two simple roots $\lambda_{1,2} \in \mathbb{C}$ of the characteristic equation (\ref{char-eq-omega}) with $m = m_0$ and ${\rm Re}(\lambda_1) > 0 > {\rm Re}(\lambda_2)$ such that 
	\begin{equation}
	\label{bifurcation}
	|\lambda_{1,2} - i \ell_0| \leq C \sqrt{|\sigma - \sigma_0|}, \qquad 
	{\rm sgn}(\omega_{m_0}''(\ell_0)) (\sigma - \sigma_0) \in (0,\varepsilon).
	\end{equation}
\end{lemma}

\begin{proof}
	For every flat band $\omega_m^{(1)} = \frac{\pi^2 m^2}{L_2^2}$, the characteristic equation (\ref{char-eq-omega}) is linear in $\lambda$ and admits only one simple root on $i \R$. Hence, we consider the non-flat band $\omega_m^{(2)}$ of Proposition \ref{prop-spectrum}. From the assumption that $\lambda = i \ell_0$ is a double root of (\ref{char-eq-omega}) at $(\sigma_0,c_0)$, we get 
	\begin{align*}
	\begin{cases} 
	\sigma_0 - c_0 \ell_0 + \omega_{m_0}(\ell_0) = 0, \\
	c_0 - \omega_{m_0}'(\ell_0) = 0, \end{cases}
	\end{align*}
	which yields (\ref{bif-point}). If $\omega_{m_0}''(\ell_0) \neq 0$, $\lambda_0 = i \ell_0$ is exactly the double root of (\ref{char-eq-omega}) for $m = m_0$. Expanding (\ref{char-eq-omega}) for $m = m_0$ in powers of $\lambda - i \ell_0$ and $\sigma - \sigma_0$ for $c = c_0$ yields:
	$$
	\sigma - \sigma_0 - \frac{1}{2} \omega_{m_0}''(\ell_0) (\lambda - i \ell_0)^2 + \mathcal{O}((\lambda - i \ell_0)^3) = 0,
	$$
	which yields (\ref{bifurcation}) for small and nonzero $|\sigma - \sigma_0|$. 
\end{proof}

\begin{example}
	\label{example-expansion}
	In the homogeneous case of $L_1 = 2 \pi$ and $L_2 = 0$, 
	the bands are given by (\ref{seq-eig}). For simplicity, let us assume that $c_0 \in (-1,1)$ so that $\ell_0 := \frac{c_0}{2} \in \mathbb{B}$ and the bifurcation happens at the lowest band $\omega_0(\ell) = \ell^2$ for $k = 0$. Then, we get the quadratic equation
	\begin{equation}
	\label{lambda-expansion}
	\sigma + i c \lambda - \lambda^2= 0,
	\end{equation}
	which admit two roots at 
	$$
	\lambda_{1,2} = \frac{i c}{2} \pm \frac{\sqrt{4 \sigma - c^2}}{2}.
	$$
	The roots are double at $\sigma = \sigma_0 = \frac{c_0^2}{4}$ for any $c = c_0 \in (-1,1)$ and satisfy (\ref{bifurcation}) for every $\sigma > \sigma_0$ since $\omega_0''(\ell_0) = 2 > 0$. Note that $c_0 = 2 \ell_0$ and $\sigma_0 = -\ell_0^2 + 2 \ell_0^2 = \ell_0^2$ are in agreement with (\ref{bif-point}).
\end{example}

\subsection{Numerical approximation of eigenvalues}

For each $m \in \mathbb{N}$, the roots of the characteristic equation (\ref{char-eq-omega}) with $\omega_m$ defined from the trace equation \eqref{char-eq} can be either obtained explicitly in the homogeneous case $L_1 = 2\pi$ and $L_2 = 0$, or obtained numerically. 

In the homogeneous case $L_1=2\pi$, $L_2=0$, if we use (\ref{seq-eig}) to re-enumerate the sequence of spectral bands, then 
the eigenvalues $\lambda$ can be computed from (\ref{char-eq-omega}) explicitly as 
\begin{equation}
\label{char-quartic}
\lambda = i \left(\frac{c}{2} - k\right) \pm \sqrt{\sigma - \frac{c^2}{4} + c k}, \quad k \in \Z.
\end{equation}
The obtained eigenvalues are shown on Figure \ref{fig:spectral-picture-L1_0} for $c = 0.5$ and $\sigma = 0.25 c^2 + \varepsilon$ 
with $\varepsilon = 0.1$. The pair of eigenvalues for $k = 0$ is shown by blue dots. These are the bifurcating eigenvalues in Example \ref{example-expansion}. All other roots are either on the purely imaginary axis in the upper half-plane for $k \geq 1$ or as symmetric complex pairs in the lower-half plane for $k \leq -1$.

\begin{figure}[htb!]
	\centering
	\includegraphics[width=12cm, height = 8.6cm]{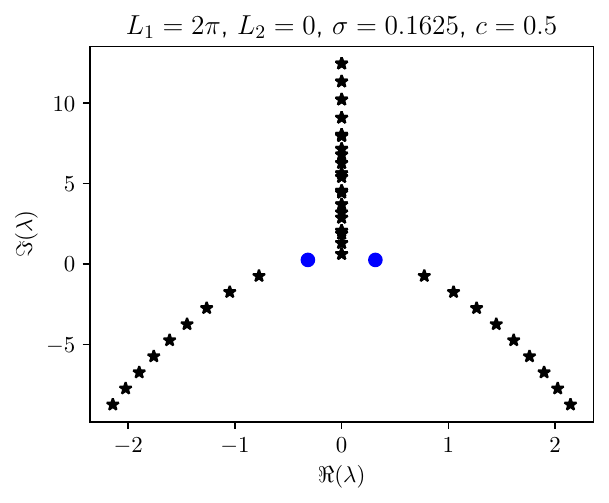}
	\caption{Eigenvalues $\lambda$ in the homogeneous case $L_1 = 2\pi$ and $L_2 = 0$. Blue dots show the pair of eigenvalues bifurcating from the origin for $k = 0$ and $\sigma > \frac{c^2}{4}$.}
	\label{fig:spectral-picture-L1_0}
\end{figure}

Several options are possible to obtain the roots numerically for $L_1 \neq 2\pi$. One may calculate the zero level curves for the real and imaginary parts, and find the eigenvalues at the intersection points. Alternatively, one may proceed by continuation from the limiting homogeneous case. In the latter case, the procedure is the following. 

First, given $K\in\mathbb N$, we compute a set of eigenvalues $(\lambda_k^0)_{k=1,\dots,K}$ for the homogeneous case $L_1 = 2\pi$ and $L_2 = 0$. Given $N\in\mathbb N$ and $L_1\in(0,2\pi)$, $L_2=2\pi-L_1$, we divide  the intervals $[L_1,2\pi]$ and $[0,L_2]$ in $N$ parts by setting $L_1^n=2\pi-n(2\pi-L_1)/N$, $L_2=nL_2/N$ for $n=0,\dots,N$. Then we compute iteratively the eigenvalues $(\lambda_k^n)_{k=1,\dots,K}$ corresponding to $(L_1^n,L_2^n)$ by solving at each $n\geq 1$ the eigenvalue equation \eqref{char-eq-omega} using the software root solver (typically \texttt{fsolve} in Matlab or Python), where the initial guess for the eigenvalues is obtained at the step $n-1$. Choosing $N$ large enough ensures the ability of the algorithm to follow the eigenvalues. Whenever two roots collide on the imaginary axis, the initial guess is given by eigenvalues perturbed off the imaginary axis on the negative and positive sides. 

The outcome of the two numerical approaches are in perfect agreement. They are represented on Figure \ref{fig:roots-lines-and-stars}. Compared to Figure \ref{fig:spectral-picture-L1_0},  some roots in the upper half-plane are also complex-valued (symmetrically about the imaginary axis).

\begin{figure}[htb!]
	\centering
	\includegraphics[width=12cm, height = 8.6cm]{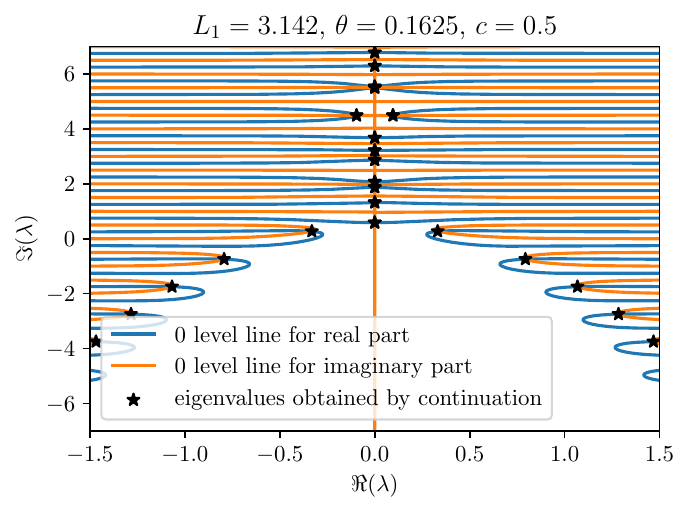}
	\caption{Numerical calculations of eigenvalues obtained from the characteristic equations \eqref{char-eq-omega} for $L_1 = L_2 = \pi$ by the two different numerical approaches.}
	\label{fig:roots-lines-and-stars}
\end{figure}

\begin{figure}[htb!]
    \centering
    \includegraphics[width=0.49\linewidth,height=6cm]{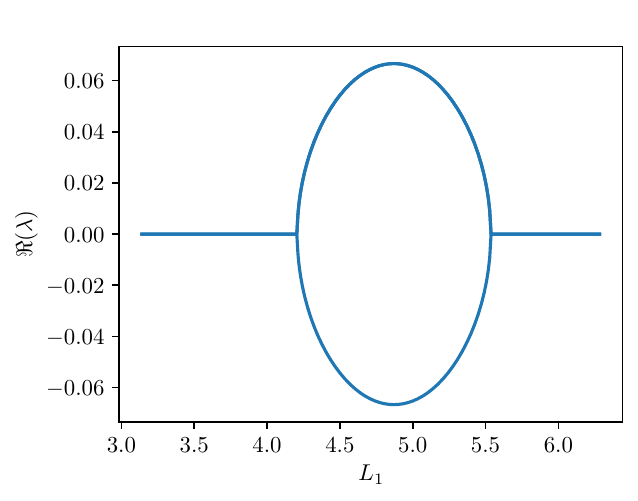}
    ~\includegraphics[width=0.49\linewidth,height=6cm]{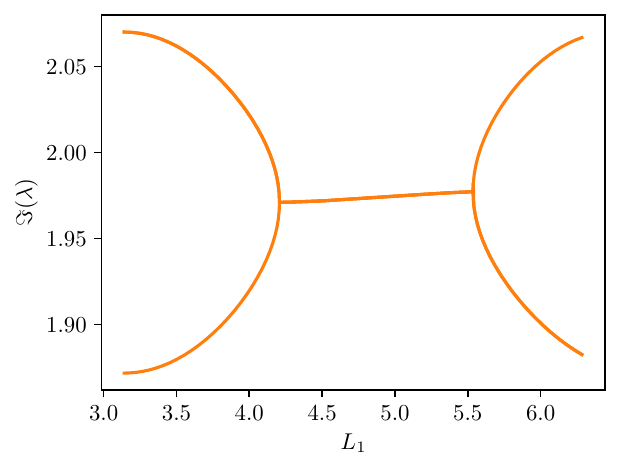}\\
    \includegraphics[width=0.49\linewidth,height=6cm]{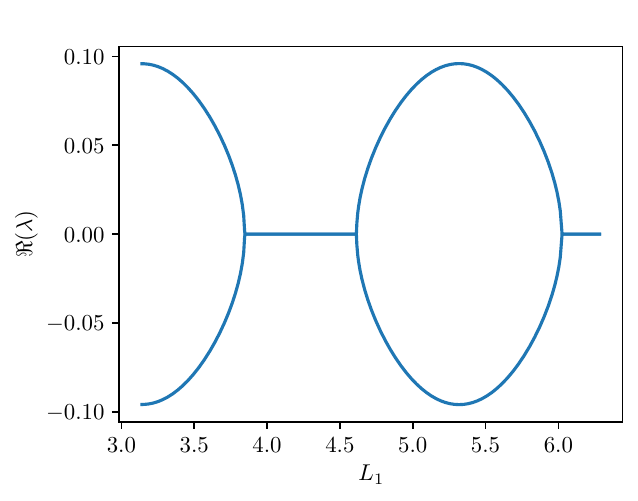}
    ~\includegraphics[width=0.49\linewidth,height=6cm]{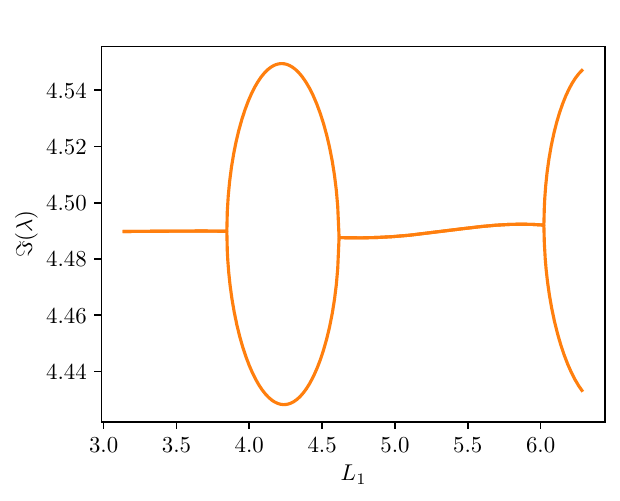}\\
    \caption{Real (left) and imaginary (right) parts of two pairs of eigenvalues coalescing in the parameter continuation in $L_1$ with $L_2 = 2\pi - L_1$.}
    \label{fig:collisions}
\end{figure}

When we continue the eigenvalues from $L_1 = 2\pi$, $L_2 = 0$ in Figure \ref{fig:spectral-picture-L1_0} 
to $L_1 = L_2 = \pi$ in Figure \ref{fig:roots-lines-and-stars}, we observe that eigenvalues starting on the imaginary axis might coalesce and thereby leave the imaginary axis. We made a focus on two of these collisions in Figure \ref{fig:collisions}. We observe that eigenvalues leave the imaginary axis after collision, but also return to the axis for a larger length $L_1$.

Following a similar numerical procedure, we also have investigated the evolution of eigenvalues $\lambda$ when $\sigma$ varies. For the sake of a better visualization, we made the study in the case $L_1=\frac{3\pi}{2}$ and $L_2=\frac\pi2$ rather than when $L_1=L_2=\pi$ where bifurcating eigenvalues are very close to one another. We have chosen to fix $c=0.5$ and to vary $\sigma$ from $0.05$ to $0.075$ (in such a way to remain close to $c^2/4$ and to isolate the candidate for the first bifurcation. The outcome is shown in Figure \ref{fig:ev_lin_sys_sigma_varies_0.050_0.075_c_0.5} (top). 

\begin{figure}[htpb!]
    \centering
    \includegraphics[width=0.49\linewidth]{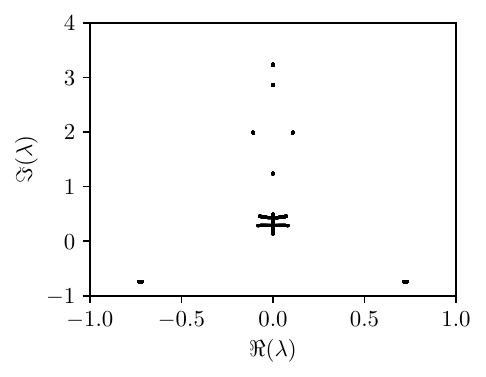}~
 \includegraphics[width=0.49\linewidth]{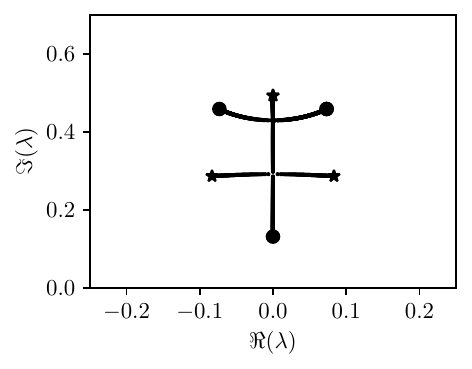} \\
     \includegraphics[width=0.49\linewidth]{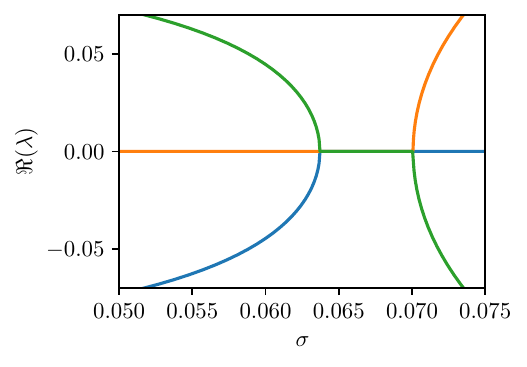}~
   \includegraphics[width=0.49\linewidth]{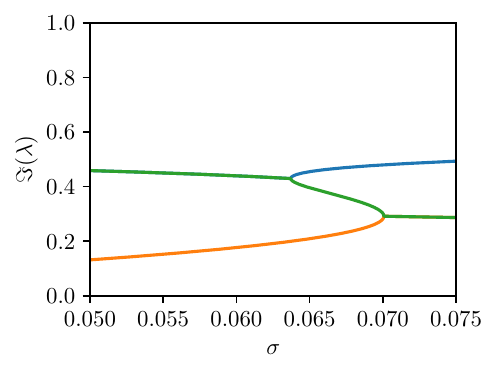}
    \caption{Top: Eigenvalues $\lambda$ when $c=0.5$ and $\sigma$ varies from $0.05$ (dots) to $0.075$ (stars) (larger picture on the left, zoomed picture on the right). Bottom: volution of the real (left) and imaginary (right) parts of the bifurcating eigenvalues with respect to $\sigma$.} \label{fig:ev_lin_sys_sigma_varies_0.050_0.075_c_0.5}
\end{figure}

We observe that two bifurcations occur for three eigenvalues located close to each other. To have a better understanding of the phenomenon, we have plotted the evolution of the real and imaginary parts of the eigenvalues involved in Figure \ref{fig:ev_lin_sys_sigma_varies_0.050_0.075_c_0.5} (bottom)  with respect to $\sigma$. For the first bifurcation, two complex eigenvalues coalesce and split along the imaginary axis. For the second bifurcation, one of these two imaginary eigenvalues and the third imaginary eigenvalue coalesce and split off the imaginary axis.

\section{Bifurcations of traveling modulating pulses}
\label{sec-7}

In this section, we develop analysis of  the spatial dynamical system (\ref{NLS-dynamics}) and give the proof of Theorem \ref{theorem-NLS-periodic}. Let us rewrite  (\ref{NLS-dynamics}) in variables $(u,r)$ instead of $(\tilde{\phi},\tilde{\eta})$ as
\begin{equation}
\label{NLS-dynamics-near}
\frac{d}{d\xi} \begin{pmatrix} u \\ r \end{pmatrix} = 
\begin{pmatrix} r \\ 
-\partial_x^2 u - 2 \partial_x r + i c r + \sigma u - 2 |u|^2 u \end{pmatrix},
\end{equation}
where $(u,r) \in \mathcal{D}$ satisfy the boundary conditions (\ref{vertex-5}) and (\ref{vertex-6}) on $\Gamma_0$. For simplification of computations, we can rewrite the system of two equations (\ref{NLS-dynamics-near}) 
with $\sigma = \sigma_0 + \varepsilon^2$ and $c = c_0$ as a scalar second-order equation:
\begin{equation}
\label{NLS-scalar-full}
-(\partial_x + \partial_{\xi})^2 u + i c_0 \partial_{\xi} u + \sigma_0 u + \varepsilon^2 u - 2 |u|^2 u  = 0.
\end{equation}

Both symmetries of the spectral problem on the necklace graph $\Gamma$ in (\ref{reduction-1}) and (\ref{reduction-2}) are invariant with respect to the time evolution of the scalar equation (\ref{NLS-scalar-full}). Hence we only consider solutions satisfying the symmetry condition in (\ref{reduction-2}), for which only eigenfunctions for the spectral bands $\{ \omega_m^{(2)}(\ell)\}_{m \in \mathbb{N}}$ in Proposition \ref{prop-spectrum} are relevant and the flat spectral bands $\{ \omega_m^{(1)}(\ell) \}_{m \in \mathbb{N}}$ can be ignored. 

We shall add the following assumption for the derivation of the normal form and for the justification analysis.

\begin{assumption}
	\label{ass-nondegeneracy}
	Pick $m_0 \in \N$ and $\ell_0 \in \mathbb{B}$ with $\omega_{m_0}''(\ell_0) > 0$ and define 
\begin{equation}
\label{bif-point-again}
\begin{cases}
\sigma_0 = -\omega_{m_0}(\ell_0) + \ell_0 \omega_{m_0}'(\ell_0), \\
c_0 = \omega_{m_0}'(\ell_0).
\end{cases}
\end{equation}
Assume that the set of purely imaginary roots $\lambda \in i \R$  of the characteristic equations 
\begin{equation}
\label{char-eq-again}
\sigma_0 + i c_0 \lambda + \omega_m(-i \lambda) = 0, \quad m \in \N,
\end{equation}
admit only simple roots, with the exception of the double root $\lambda = i \ell_0$ for $m = m_0$.
\end{assumption}

Formal derivation of the normal form is given in Section \ref{sec-7a}. Justification of the normal form can be found in Section \ref{sec-7b}.

\subsection{Formal derivation of the normal form}
\label{sec-7a}

By differentiating the spectral problem (\ref{spectral-prob}) in $\ell$ at $\ell_0$, we obtain for $m = m_0$:
\begin{align}
\MoveEqLeft
-(\partial_x + i \ell_0)^2 f_{m_0}(\ell_0,x) = \omega_{m_0}(\ell_0) f_{m_0}(\ell_0,x), 
\label{rel-1} \\
\MoveEqLeft
 \begin{multlined}[0.8\textwidth]
     -(\partial_x + i \ell_0)^2 \partial_{\ell} f_{m_0}(\ell_0,x) - 2 i (\partial_x + i \ell_0) f_{m_0}(\ell_0,x) = \\
  \omega_{m_0}(\ell_0) \partial_{\ell} f_{m_0}(\ell_0,x)  
 + \omega_{m_0}'(\ell_0) f_{m_0}(\ell_0,x),
 \end{multlined}
  \label{rel-2}
  \\
  \MoveEqLeft
  \begin{multlined}     
  -(\partial_x + i \ell_0)^2 \partial^2_{\ell} f_{m_0}(\ell_0,x) - 4 i (\partial_x + i \ell_0) \partial_{\ell} f_{m_0}(\ell_0,x) + 2 f_{m_0}(\ell_0,x) 
\\
 = \omega_{m_0}(\ell_0) \partial^2_{\ell} f_{m_0}(\ell_0,x)  + 2 \omega_{m_0}'(\ell_0) \partial_{\ell} f_{m_0}(\ell_0,x) + \omega_{m_0}''(\ell_0) f_{m_0}(\ell_0,x), 
\end{multlined}
\label{rel-3}
\end{align}
thanks to smoothness of the band function $\omega_{m_0}(\ell)$ and the Bloch function $f(\ell,x)$. Let us equip the Hilbert space $L^2(\Gamma_0)$ with the standard inner product $\langle \cdot, \cdot \rangle$. By taking the inner product of (\ref{rel-2}) and (\ref{rel-3}) with $f_{m_0}(\ell_0,\cdot)$ and using (\ref{rel-1}), we obtain 
\begin{align}
\omega_{m_0}'(\ell_0) = \frac{-2i \langle (\partial_x + i \ell_0) f_{m_0}(\ell_0,\cdot), f_{m_0}(\ell_0,\cdot)\rangle}{\| f_{m_0}(\ell_0,\cdot)\|^2_{L^2}} 
\label{inner-1} 
\end{align}
and
\begin{align}
& \frac{1}{2} \omega_{m_0}''(\ell_0) = 1 
\notag
\\
& \qquad - \frac{2i \langle (\partial_x + i \ell_0) \partial_{\ell} f_{m_0}(\ell_0,\cdot), f_{m_0}(\ell_0,\cdot)\rangle  + \omega_{m_0}'(\ell_0) \langle \partial_{\ell} f_{m_0}(\ell_0,\cdot), f_{m_0}(\ell_0,\cdot)\rangle}{\| f_{m_0}(\ell_0,\cdot)\|^2_{L^2}}. \label{inner-2}
\end{align}

We adopt Assumption \ref{ass-nondegeneracy} for the bifurcation point $(\sigma_0,c_0)$ of Lemma \ref{prop-expansion}  and let $\varepsilon > 0$ be a small parameter which defines the perturbation of $\sigma = \sigma_0 + \varepsilon^2$. The normal form for the bifurcation at $(\sigma_0,c_0)$ is derived by using the following formal expansion:
\begin{equation}
\label{formal-expansion}
u = \varepsilon u_1 + \varepsilon^2 u_2 + \varepsilon^3 u_3 + \varepsilon^4 R_{\varepsilon},
\end{equation}
where the leading-order term is written in the separable form 
\begin{equation}
    \label{first-term}
    u_1(\xi,x) = A(\varepsilon \xi) e^{i \ell_0 \xi} f_{m_0}(\ell_0,x),
\end{equation}
with the amplitude $A(\varepsilon \xi)$ to be determined. We substitute 
(\ref{formal-expansion}) into (\ref{NLS-scalar-full}) to get 
equations for $u_2$ and $u_3$ in powers of $\varepsilon^2$ and $\varepsilon^3$. This gives the leading-order approximation for the bifurcation provided that the remainder term 
$R_{\varepsilon}$ is bounded in $L^{\infty}(\mathbb{R} \times \Gamma)$ for sufficiently small $\varepsilon > 0$. The remainder term is to be estimated in Section \ref{sec-7b}.

Substitution of (\ref{formal-expansion}) into (\ref{NLS-scalar-full}) yields 
\begin{align*}
& \varepsilon \left[ -(\partial_x + \partial_{\xi})^2 u_1 + i c_0 \partial_{\xi} u_1 + \sigma_0 u_1 \right] \\
& \qquad + \varepsilon^2 \left[ -(\partial_x + \partial_{\xi})^2 u_2 + i c_0 \partial_{\xi} u_2 + \sigma_0 u_2 \right] \\
& \qquad + \varepsilon^3 \left[ -(\partial_x + \partial_{\xi})^2 u_3 + i c_0 \partial_{\xi} u_3 + \sigma_0 u_3  + u_1 - 2 |u_1|^2 u_1 \right] + \mathcal{O}(\varepsilon^4) = 0
\end{align*}
The order of $\mathcal{O}(\varepsilon)$ is satisfied by (\ref{first-term}) thanks to the definition of $\sigma_0$ in (\ref{bif-point-again}) and the spectral problem (\ref{rel-1}). The order of $\mathcal{O}(\varepsilon^2)$ yields the linear inhomogeneous equation
\begin{align}
& -(\partial_x + \partial_{\xi})^2 u_2 + i c_0 \partial_{\xi} u_2 + \sigma_0 u_2  \notag \\
& \qquad = A'(\varepsilon \xi) e^{i \ell_0 \xi}  \left[ 2 \partial_x f_{m_0}(\ell_0,x) + 2i \ell_0 f_{m_0}(\ell_0,x) - i c_0 f_{m_0}(\ell_0,x) \right].
\label{NLS-second-level}
\end{align}
The solution at the leading-order is given in the separable form:
\begin{equation}
\label{second-term}
u_2(\xi,x) = -i A'(\varepsilon \xi) e^{i \ell_0 \xi} \partial_{\ell} f_{m_0}(\ell_0,x),
\end{equation}
thanks to the definition of $(\sigma_0,c_0)$ in (\ref{bif-point-again}) and the derivative equation (\ref{rel-2}). Finally, the 
order of $\mathcal{O}(\varepsilon^3)$ yields the linear inhomogeneous equation
\begin{align}
&-(\partial_x + \partial_{\xi})^2 u_3 + i c_0 \partial_{\xi} u_3 + \sigma_0 u_3   \notag \\
& \qquad = A''(\varepsilon \xi) e^{i \ell_0 \xi} \left[ f_{m_0}(\ell_0,x) - 2 i \partial_x \partial_{\ell} f_{m_0}(\ell_0,x) + 2\ell_0 \partial_{\ell} f_{m_0}(\ell_0,x) - c_0 \partial_{\ell} f_{m_0}(\ell_0,x) \right]  \label{NLS-third-level} \\
& \qquad - A(\varepsilon \xi) e^{i \ell_0 \xi} f_{m_0}(\ell_0,x) + 2 |A(\varepsilon \xi)|^2 A(\varepsilon \xi) e^{i \ell_0 \xi} |f_{m_0}(\ell_0,x)|^2 f_{m_0}(\ell_0,x). \notag
\end{align}
By using the second derivative equation (\ref{rel-3}), we represent the solution in the form:
\begin{equation}
    \label{third-term}
    u_3(\xi,x) = -\frac{1}{2} A''(\varepsilon \xi) e^{i \ell_0 \xi} \partial^2_{\ell} f_{m_0}(\ell_0,x) + \tilde{u}_3(\xi,x)
\end{equation}
and rewrite the linear inhomogeneous equation (\ref{NLS-third-level}) at the leading order in the equivalent form:
\begin{multline}
    -(\partial_x + \partial_{\xi})^2 \tilde{u}_3 + i c_0 \partial_{\xi} \tilde{u}_3 + \sigma_0 \tilde{u}_3  = \frac{1}{2} \omega_{m_0}''(\ell_0) A''(\varepsilon \xi) e^{i \ell_0 \xi} f_{m_0}(\ell_0,x) \\  - A(\varepsilon \xi) e^{i \ell_0 \xi} f_{m_0}(\ell_0,x) + 2 |A(\varepsilon \xi)|^2 A(\varepsilon \xi) e^{i \ell_0 \xi} |f_{m_0}(\ell_0,x)|^2 f_{m_0}(\ell_0,x).
\label{NLS-third-level-again}
\end{multline}
By Fredholm's theorem, a solution $\tilde{u}_3(\xi,\cdot)$ exists in $H^2_{\rm NK}(\Gamma_0)$ for $\xi \in \mathbb{R}$ 
if and only if the right-hand side is orthogonal to $e^{i \ell_0 \xi} f_{m_0}(\ell_0,x)$, which is the generator of the kernel of the 
linear operator at the left-hand-side of (\ref{NLS-third-level-again}). Projecting the right-hand side of (\ref{NLS-third-level-again}) to $e^{i \ell_0 \xi} f_{m_0}(\ell_0,x)$ leads to the normal form equation 
\begin{equation}
    \label{normal-form}
    \frac{1}{2} \omega_{m_0}''(\ell_0) A'' - A + 2 \gamma |A|^2 A = 0, \qquad \gamma := \frac{\| f_{m_0}(\ell_0,\cdot) \|_{L^4(\Gamma_0)}^4}{\| f_{m_0}(\ell_0,\cdot) \|_{L^2(\Gamma_0)}^2}.
\end{equation}
If the solvability condition (\ref{normal-form}) is satisfied, there exists a solution $\tilde{u}_3(\xi,\cdot) \in H^2_{\rm NK}(\Gamma_0)$ to the linear inhomogeneous equation (\ref{NLS-third-level-again}) up to the leading order in the separable form
\begin{equation}
    \label{third-term-again}
    \tilde{u}_3(\xi,x) = 2 |A(\varepsilon \xi)|^2 A(\varepsilon \xi) e^{i \ell_0 \xi} h(x),
\end{equation}
where $h(x)$ is the solution of the linear inhomogeneous equation 
\begin{equation}
    \label{third-term-equation}
-(\partial_x + i \ell_0)^2 h - \omega_{m_0}(\ell_0) h  = |f_{m_0}(\ell_0,x)|^2 f_{m_0}(\ell_0,x) - \gamma f_{m_0}(\ell_0,x).
\end{equation}
The solution $h \in H^2_{\rm NK}(\Gamma_0)$ to (\ref{third-term-equation}) is uniquely defined under the orthogonality condition 
$$
\langle h, f_{m_0}(\ell_0,\cdot) \rangle = 0.
$$

\begin{remark}
The main outcome of the normal form equation (\ref{normal-form}) is that a homoclinic orbit exists for $\omega_{m_0}''(\ell_0) > 0$ since $\gamma > 0$. If the bifurcation of complex eigenvalues in Lemma \ref{prop-expansion} occurs for $\sigma < \sigma_0$ with $\omega_{m_0}''(\ell_0) < 0$, then a homoclinic orbit does not exist for $\sigma < \sigma_0$ since the cubic nonlinear term in the normal form equation (\ref{normal-form}) does not change the sign, whereas the two linear terms change the sign. This gives a criterion to distinguish bifurcations which lead to traveling modulating pulses from those which do not lead to traveling modulating pulses.  
\end{remark}

\begin{example}
    Consider the two splittings of double eigenvalues shown in Figure \ref{fig:ev_lin_sys_sigma_varies_0.050_0.075_c_0.5} for $L_1 = \frac{3 \pi}{2}$ and $L_2 = \frac{\pi}{2}$. When $\sigma$ is increased, the first splitting shows the transition from the complex to purely imaginary eigenvalues, hence $\omega''_{m_0}(\ell_0) < 0$ for this case with no homoclinic orbits from the normal form equation (\ref{normal-form}). On the other hand, the second splitting shows the transition from purely imaginary to complex eigenvalues, hence $\omega''_{m_0}(\ell_0) > 0$ for this case with a homoclinic orbit from the normal form equation (\ref{normal-form}). Note that for $L_1 = 2\pi$ and $L_2 = 0$, only the latter bifurcation is possible according to Example \ref{example-expansion} since $\omega_{m_0}''(\ell_0) > 0$ holds in the homogeneous case. 
\end{example}

\subsection{Justification of the normal form}
\label{sec-7b}

Based on the formal derivation of the normal form (\ref{normal-form}) in Section \ref{sec-7a}, we adopt the orthogonal decomposition 
\begin{equation}
\label{decomposition-orth}
u(\xi,x) = \varepsilon e^{i \ell_0 \xi} \left[ A_{\varepsilon}(\varepsilon \xi) f_{m_0}(\ell_0,x) + v(\xi,x)\right], \quad \langle v(\xi,\cdot), f_{m_0}(\ell_0,\cdot) \rangle = 0, 
\end{equation}
where $A_{\varepsilon}$ and $v$ are to be determined. The scalar equation (\ref{NLS-scalar-full}) is rewritten in the equivalent form:
\begin{multline}
     \left[ \mathcal{L} - \omega_{m_0}(\ell_0) \right] v + \left[ i \omega_{m_0}'(\ell_0) - 2(\partial_x + i \ell_0) \right] \partial_{\xi} v - \partial_{\xi}^2 v  \\
 = \varepsilon A_{\varepsilon}'  \left[ 2 (\partial_x + i \ell_0) f_{m_0} - i \omega_{m_0}'(\ell_0) f_{m_0} \right]   \\
+ \varepsilon^2 A_{\varepsilon}'' f_{m_0} + \varepsilon^2 \left[ 2 |A_{\varepsilon} f_{m_0} + v|^2 - 1 \right] (A_{\varepsilon} f_{m_0} + v),
\label{dynamics}
\end{multline}
where $\mathcal{L} := -(\partial_x + i \ell_0)^2$ is the Schr\"{o}dinger operator considered in $L^2(\Gamma_0)$ with the domain in $H^2_{\rm NK}(\Gamma_0)$. We complete the scalar equation (\ref{dynamics}) with the projection equation to $f_{m_0}$ given by 
\begin{align}
\label{projection}
-2 \frac{d}{d \xi} \langle \partial_x v, f_{m_0} \rangle 
= \varepsilon^2 (A_{\varepsilon}'' - A_{\varepsilon}) \| f_{m_0} \|^2_{L^2} + 2 \varepsilon^2 \langle 
|A_{\varepsilon} f_{m_0} + v|^2 (A_{\varepsilon} f_{m_0} + v), f_{m_0} \rangle, 
\end{align}
where we have used (\ref{inner-1}) to remove the $\mathcal{O}(\varepsilon)$ term. We are now ready to proceed with analysis of the system (\ref{dynamics}) and (\ref{projection}) which gives the proof of Theorem \ref{theorem-NLS-periodic}.

\subsubsection{Near-identity transformations}

It follows from (\ref{second-term}) that the near-identity transformation given by 
\begin{equation}
\label{transformation-1}
v(\xi,x) = -i \varepsilon A_{\varepsilon}'(\varepsilon \xi) \left[ \partial_{\ell} f_{m_0}(\ell_0,x) - \frac{\langle \partial_{\ell} f_{m_0}, f_{m_0} \rangle}{\| f_{m_0}\|^2_{L^2}} f_{m_0}(\ell_0,\cdot) \right] 
+ \tilde{v}(\xi,x),
\end{equation}
where $\tilde{v}$ satisfies the orthogonality condition $\langle \tilde{v}(\xi,\cdot), f_{m_0}(\ell_0,\cdot) \rangle = 0$, transforms (\ref{projection}) into 
\begin{align}
& \varepsilon^2 \left( \frac{1}{2} \omega_{m_0}''(\ell_0) A_{\varepsilon}'' - A_{\varepsilon} 
\right) \| f_{m_0} \|^2_{L^2} + 2 \varepsilon^2 \langle 
|A_{\varepsilon} f_{m_0} + v|^2 (A_{\varepsilon} f_{m_0} + v), f_{m_0} \rangle \notag \\
& \qquad = -2 \frac{d}{d \xi} \langle \partial_x \tilde{v}, f_{m_0} \rangle,
\label{normal-form-perturbed}
\end{align}
where we have used (\ref{inner-2}) to simplify the $\mathcal{O}(\varepsilon^2)$ term.  Truncation of $v$ in the last term of the left-hand side of (\ref{normal-form-perturbed}) and $\tilde{v}$ in the right-hand-side of (\ref{normal-form-perturbed}) yields formally the normal-form equation (\ref{normal-form}). 

Transformation (\ref{transformation-1}) removes the $\mathcal{O}(\varepsilon)$ terms in equation (\ref{dynamics}) and satisfies the constraint 
$\langle v(\xi,\cdot), f_{m_0}(\ell_0,\cdot) \rangle = 0$. Furthermore, it follows from (\ref{third-term}) and (\ref{third-term-again}) that the second near-identity transformation 
\begin{align}
\tilde{v}(\xi,x) &= -\frac{1}{2} \varepsilon^2 A_{\varepsilon}''(\varepsilon \xi) \left[ \partial^2_{\ell} f_{m_0}(\ell_0,x) - \frac{\langle \partial^2_{\ell} f_{m_0}, f_{m_0} \rangle}{\| f_{m_0}\|^2_{L^2}} f_{m_0}(\ell_0,\cdot) \right]  \notag \\
& \quad 
+ 2 |A_{\varepsilon}(\varepsilon \xi)|^2 A_{\varepsilon}(\varepsilon \xi) h(x) 
+ \tilde{\tilde{v}}(\xi,x),\label{transformation-2}
\end{align}
with $h$ solving (\ref{third-term-equation}) and $\tilde{\tilde{v}}$ satisfying the orthogonality condition $\langle \tilde{\tilde{v}}(\xi,\cdot), f_{m_0}(\ell_0,\cdot) \rangle = 0$
removes the $\mathcal{O}(\varepsilon^2)$ terms in equation (\ref{dynamics}) and satisfies the constraint 
$\langle v(\xi,\cdot), f_{m_0}(\ell_0,\cdot) \rangle = 0$.

After the near-identity transformations (\ref{transformation-1}) and (\ref{transformation-2}), 
the right-hand-side of equation (\ref{dynamics}) truncated at $\tilde{\tilde{v}} = 0$ is of the formal order of $\mathcal{O}(\varepsilon^3)$.
Hence, we rewrite equation (\ref{dynamics}) in the abstract form:
\begin{align}
& \left[ \mathcal{L} - \omega_{m_0}(\ell_0) \right] v + \left[ i \omega_{m_0}'(\ell_0) - 2(\partial_x + i \ell_0) \right] \partial_{\xi} v - \partial_{\xi}^2 v \notag \\
& \quad = -\varepsilon^2 v + \varepsilon^3 H + \varepsilon^2 N(v),
\label{dynamics-trans}
\end{align}
where we dropped the tilde notations for $v$ and introduced the non-homogeneous term $H$ and the nonlinear term $N(v)$ with the following properties. 

\begin{lemma}
	\label{lem-nonlinear-1}
	Assume that $A_{\varepsilon}$ is a smooth homoclinic orbit in the variable $\varepsilon \xi$ such that 
	\begin{align}
	\label{C-0}
	\sup_{\xi \in \R} \left(|A_{\varepsilon}(\varepsilon \xi)| + |A_{\varepsilon}'(\varepsilon \xi)| + |A_{\varepsilon}''(\varepsilon \xi)|\right) \leq C_0 < \infty.
\end{align}
Then, $H$ and $N(v)$ in (\ref{dynamics-trans}) satisfy the bounds 
	\begin{align}
	\| H(\xi,\cdot) \|_{H^2_{\rm NK}} \leq C \left( |A_{\varepsilon}'''(\varepsilon \xi)| + |A_{\varepsilon}''''(\varepsilon \xi)| \right) \label{bound-1} 
    \end{align}
    and
    \begin{align}
& \| N(v)(\xi,\cdot) \|_{H^2_{\rm NK}} \leq C \left( |A_{\varepsilon}(\varepsilon \xi)| + |A_{\varepsilon}'(\varepsilon \xi)| + |A_{\varepsilon}''(\varepsilon \xi)| \right) \| v(\xi,\cdot) \|_{H^2_{\rm NK}} (1 + \| v(\xi,\cdot) \|_{H^2_{\rm NK}}) \notag \\ 
& \qquad + C \| v(\xi,\cdot) \|_{H^2_{\rm NK}}^3,\label{bound-2}
	\end{align}
	where $C > 0$ is a generic constant which only depends on $C_0$ in (\ref{C-0}).
\end{lemma}

\begin{proof}
	Bounds (\ref{bound-1}) and (\ref{bound-2}) follow by substituting  (\ref{transformation-1}) and (\ref{transformation-2}) into equation (\ref{dynamics}) since all terms containing $A_{\varepsilon}$ are written in the product form of functions of $\varepsilon \xi$ and $x$, where functions of $\varepsilon \xi$ are smooth and functions of $x$ belong to $H^2_{\rm NK}(\Gamma_0)$. 
\end{proof}

\subsubsection{Construction of a local stable-center manifold}

By Lemma \ref{lem-linearized}, the left-hand-side of equation (\ref{dynamics-trans}) can be solved 
as a linear superposition of normal modes:
\begin{equation}
\label{decomposition-near}
v(\xi,x) = \sum_{\lambda_k \in \Lambda \backslash \{ i \ell_0 \}} v_k(\xi) f_m(-i\lambda_k,x),  
\end{equation}
where $\Lambda$ denotes the set of admissible values of $\lambda$ from roots of the characteristic equations (\ref{char-eq-again}) at the bifurcation point $(\sigma,c) = (\sigma_0,c_0)$ for some $\omega_m(-i\lambda)$ found from roots of the characteristic equation (\ref{char-eq}). The double root $\lambda = i \ell_0$ which exists by Lemma \ref{prop-expansion} is excluded from the decomposition (\ref{decomposition-near}) since the corresponding mode is included in the $A_{\varepsilon}$ term of the decomposition (\ref{decomposition-orth}). 

We further define the decomposition 
$$
\Lambda = \{ i \ell_0 \} \cup \Lambda_c \cup \Lambda_s \cup \Lambda_u,
$$ 
where $\Lambda_c$ includes roots with ${\rm Re}(\lambda) = 0$, $\Lambda_s$ includes roots with ${\rm Re}(\lambda) < 0$, and $\Lambda_u$ includes roots with ${\rm Re}(\lambda) > 0$. By Assumption \ref{ass-nondegeneracy}, all roots in $\Lambda_c$ are simple. 

Let us use the decomposition (\ref{decomposition-near}) to represent the solution 
of the initial-value problem
\begin{equation}
    \label{initial-value}
\begin{cases} 
[ \mathcal{L} - \omega_{m_0}(\ell_0) ] v + [ i \omega_{m_0}'(\ell_0) - 2(\partial_x + i \ell_0) ] \partial_{\xi} v - \partial_{\xi}^2 v = 0, \\
(v, \partial_{\xi} v) |_{\xi = 0} = {\bf a}, \end{cases}
\end{equation}
as $v(\xi) = S(\xi) {\bf a}$, where the dependence of $x$ is dropped to make notations easy. Restrictions of $S(\xi)$ to $\Lambda_{c,s,u}$ are then denoted as $S_{c,s,u}(\xi)$, respectively. 
Assumption \ref{ass-nondegeneracy} can be rewritten in the form used for construction of the local stable-center manifold. 

\begin{assumption}
	\label{ass-semigroup}
	There exists $K > 0$ such that 
	\begin{align*}
	\| S_c(\xi) \|_{H^2_{\rm NK} \to H^2_{\rm NK}} \leq K, \quad & \xi \in \R, \\
	\| S_s(\xi) \|_{H^2_{\rm NK} \to H^2_{\rm NK}} \leq K, \quad & \xi \geq 0, \\
	\| S_u(\xi) \|_{H^2_{\rm NK} \to H^2_{\rm NK}} \leq K, \quad & \xi \leq 0.
	\end{align*}
\end{assumption}

The local center-stable manifold is given by the following lemma. 

\begin{lemma}
\label{lem-nonlinear-2}
	Assume that $A_{\varepsilon}$ is a smooth homoclinic orbit satisfying 
\begin{equation}
\label{C-0-extended}
\sup_{\xi \in \mathbb{R}} \sum_{k=0}^{2} |A_{\varepsilon}^{(k)}(\varepsilon \xi)| +	\varepsilon \int_{\mathbb{R}} \sum_{k=0}^{4} |A_{\varepsilon}^{(k)}(\varepsilon \xi)| d \xi \leq C_0 < \infty,
\end{equation}
with a constant $C_0$ independent of $\varepsilon > 0$ and that Assumption \ref{ass-semigroup} is satisfied. For every $N_0 > 0$, there exist $\varepsilon_0 > 0$ and $C > 0$ such that for every $\varepsilon \in (0,\varepsilon_0)$, there exists a solution to (\ref{dynamics-trans}) satisfying 
\begin{equation}
\label{bound-on-v}
\sup_{\xi \in [0,N_0 \varepsilon^{-1}]} \| v(\xi,\cdot) \|_{H^2_{\rm NK}} \leq C \varepsilon^2, \quad \sup_{\xi \in [0,N_0 \varepsilon^{-1}]} \| \partial_{\xi} v(\xi,\cdot) \|_{H^2_{\rm NK}} \leq C \varepsilon^3.
\end{equation}
\end{lemma}

\begin{proof}
	Let $\Pi_{c,s,u}$ be projection operators defined by the restriction of $S(\xi)$ to $\Lambda_{c,s,u}$. Hence, we write $v_{c,s,u}(\xi,\cdot) = \Pi_{c,s,u} v(\xi,\cdot)$. By using Assumption \ref{ass-semigroup}, we obtain the local center-stable manifold in components $v_{c,s,u}(\xi,\cdot) \in H^2_{\rm NK}(\Gamma_0)$ from solutions of the following system of integral equations, which follow from (\ref{dynamics-trans}):
\begin{align*}
v_c(\xi,\cdot) &= S_c(\xi) {\bf a} + \int_0^{\xi} S_c(\xi-\xi') \left[ -\varepsilon^2 v_c(\xi',\cdot) + \varepsilon^3 \Pi_c H(\xi',\cdot) + \varepsilon^2 \Pi_c N(v(\xi',\cdot)) \right] d\xi', \\
v_s(\xi,\cdot) &= S_s(\xi) {\bf b} - \int_{\xi}^{\xi_0} S_s(\xi-\xi') \left[ -\varepsilon^2 v_s(\xi',\cdot) + \varepsilon^3 \Pi_s H(\xi',\cdot) + \varepsilon^2 \Pi_s N(v(\xi',\cdot)) \right] d\xi', \\
v_u(\xi,\cdot) &= - \int_{\xi}^{\xi_0} S_u(\xi-\xi') \left[ -\varepsilon^2 v_u(\xi',\cdot) + \varepsilon^3 \Pi_u H(\xi',\cdot) + \varepsilon^2 \Pi_u N(v(\xi',\cdot)) \right] d\xi', 
\end{align*}
where $\xi_0 := N_0 \varepsilon^{-1}$ and the parameters ${\bf a} := (v_c(0,\cdot),\partial_{\xi} v_c(0,\cdot)) \in H^2_{\rm NK}(\Gamma_0) \times H^2_{\rm NK}(\Gamma_0)$, ${\bf b} := [S_s(\xi_0)]^{-1} (v_s(\xi_0,\cdot),\partial_{\xi} v_s(\xi_0,\cdot))  \in H^2_{\rm NK}(\Gamma_0) \times H^2_{\rm NK}(\Gamma_0)$ are at our disposal. 

By using the bounds (\ref{bound-1}) and (\ref{bound-2}) of Lemma \ref{lem-nonlinear-1}, the bounds of Assumption \ref{ass-semigroup}, and Gronwall's inequality, we obtain estimates for the solutions of the integral equations:
\begin{align*}
& \sup_{\xi \in [0,\xi_0]} \| v(\xi,\cdot) \|_{H^2_{\rm NK}} \leq 
K \left( \| {\bf a} \|_{H^2_{\rm NK}} + \| {\bf b} \|_{H^2_{\rm NK}}
+ \varepsilon^3 C \int_0^{\xi_0} \left( |A_{\varepsilon}'''(\varepsilon \xi')| + |A_{\varepsilon}''''(\varepsilon \xi')| \right) d\xi' \right) \\
& \qquad 
+ \varepsilon^2 K \left( \sup_{\xi \in [0,\xi_0]} \| v(\xi,\cdot) \|_{H^2_{\rm NK}} \right) \left( 1 + C \int_0^{\xi_0} \left( |A_{\varepsilon}(\varepsilon \xi')| + |A_{\varepsilon}'(\varepsilon \xi')| + |A_{\varepsilon}''(\varepsilon \xi')|\right) d\xi' \right)  \\
& \qquad 
+ \varepsilon^2 K C \xi_0 \| v(\xi,\cdot) \|^3_{H^2_{\rm NK}}.
\end{align*}
Since $\varepsilon^2 \xi_0 = \varepsilon N_0$ is small and the bound (\ref{C-0-extended}) is assumed, the restriction 
$$
\| {\bf a} \|_{H^2_{\rm NK}} + \| {\bf b} \|_{H^2_{\rm NK}} \leq C_0 \varepsilon^2
$$ 
implies the first bound in   (\ref{bound-on-v}). The second bound in (\ref{bound-on-v}) follows from the fact that $H(\xi,x)$ in (\ref{dynamics-trans}) is written as a sum of products of functions of $\varepsilon \xi$ and $x$. Hence, the derivative of $v(\xi,\cdot)$ in $\xi$ is one order smaller in powers of $\varepsilon$.
\end{proof}

\subsubsection{Finding a homoclinic orbit under the reversibility constraint}

We impose the following reversibility constraints on the solutions of the spatial dynamical system (\ref{NLS-dynamics-near}):
\begin{equation}
\label{reversibility}
{\rm Im}(u)(0,x_0) = 0, \qquad {\rm Re}(r)(0,x_0) = 0,
\end{equation}
or equivalently, the constraint on solutions of the scalar equation (\ref{NLS-scalar-full}):
\begin{equation}
\label{reversibility-scalar}
{\rm Im}(u)(0,x_0) = 0, \quad \partial_{\xi} {\rm Re}(u)(0,x_0) = 0,
\end{equation}
where $x_0 \in \{\frac{1}{2} L_1, L_1 + \frac{1}{2} L_2\}$ is a point of symmetry of $\Gamma$ located in the basic cell $\Gamma_0$. The reversibility constraint (\ref{reversibility}) picks the specific symmetric solutions of (\ref{NLS-dynamics-near}) satisfying 
\begin{equation}
\label{reversion}
\left\{ (u,r) \in C^1(\mathbb{R},\mathcal{D}) : \;\;
\begin{cases} {\rm Im}(u)(-\xi,x_0-x) = -{\rm Im}(u)(\xi,x-x_0), \\ {\rm Re}(r)(-\xi,x_0-x) = -{\rm Re}(r)(\xi,x-x_0) \end{cases} \right\},
\end{equation}
or equivalently, 
\begin{equation}
\label{reversion-scalar}
\left\{ u \in C^1(\mathbb{R},H^2_{\rm NK}) : \;\; 
\begin{cases} {\rm Im}(u)(-\xi,x_0-x) = -{\rm Im}(u)(\xi,x-x_0), \\ \partial_{\xi} {\rm Re}(u)(-\xi,x_0-x) = -\partial_{\xi} {\rm Re}(u)(\xi,x-x_0) \end{cases} \right\},
\end{equation}
If the solution $u(\xi,x)$ satisfies (\ref{reversion-scalar}) and $A_{\varepsilon}(\xi)$ belongs to 
\begin{equation}
\label{reversion-A}
X_R := \left\{ A \in C^1(\mathbb{R}) : \;\; 
\begin{cases} 
{\rm Im}(A)(-\xi) = -{\rm Im}(A)(\xi), \\ {\rm Re}(A)(-\xi) = {\rm Re}(A)(\xi) \end{cases}  \right\},
\end{equation}
then the decomposition (\ref{decomposition-orth}) implies that $v(\xi,x)$ satisfies (\ref{reversion-scalar}) because the eigenfunction $w(\xi,x) := e^{i \ell_0 \xi} f(\ell_0,x)$ satisfies the reversibility constraint $\bar{w}(-\xi,x_0-x) = w(\xi,x-x_0)$. 

Both equations (\ref{dynamics}) and (\ref{projection}) inherit the same reversion symmetry (\ref{reversion-scalar}) and (\ref{reversion-A}), respectively. Moreover, for any given $A_{\varepsilon} \in H^{\infty}(\mathbb{R})$ satisfying (\ref{C-0-extended}), we can use the cut-off operators as in \cite{DPS} to cut the support of $A_{\varepsilon}$ on $[0,\xi_0]$ and construct the solution of Lemma \ref{lem-nonlinear-2} satisfying the symmetry (\ref{reversion-scalar}). With this decomposition, we perform the final step of obtaining solutions of the projection equation (\ref{projection}). 

\begin{lemma}
	\label{lem-nonlinear-3}
	Let us define the mapping $A_{\varepsilon} \to  v_{\varepsilon} \in C^1(\mathbb{R},H^2_{\rm NK}(\Gamma_0))$ by the solution of Lemma \ref{lem-nonlinear-2} extended from $[0,\xi_0]$ to $[-\xi_0,\xi_0]$ according to the reversion constraint (\ref{reversion-scalar}) and concatanated by any arbitrary smooth function bounded on $(-\infty,-\xi_0] \cup [x_0,\infty)$. Then, there exists a smooth homoclinic solution $A_{\varepsilon}$ to (\ref{projection}) 
satisfying (\ref{C-0-extended}) and (\ref{reversion-A}).  Moreover, 
there is $\varepsilon$-independent constant $C > 0$ such that 
$\| A_{\varepsilon} - A \|_{L^{\infty}} \leq C \varepsilon$, where 
$A$ is a smooth homoclinic orbit of the normal form equation (\ref{normal-form}) satisfying (\ref{reversion-A}). 
\end{lemma}

\begin{proof}
	Due to the first near-identity transformation (\ref{transformation-1}), the projection equation (\ref{projection}) is transformed to the form (\ref{normal-form-perturbed}) with the leading order given by the normal form equation (\ref{normal-form}). If $\omega_{m_0}''(\ell_0) > 0$, then there exists a smooth homoclinic orbit $A \in H^{\infty}(\mathbb{R})$ satisfying the reversion symmetry (\ref{reversion-A}). The normal form equation (\ref{normal-form-perturbed}) is set for $A_{\varepsilon} \in H^2(\mathbb{R})$ on the real line $\mathbb{R}$. The linearized operator at the leading-order term $A \in H^{\infty}(\mathbb{R})$ is a diagonal composition of two Schr\"{o}dinger operators 
	\begin{align*}
	\mathcal{L}_+ &= -\frac{1}{2} \omega_{m_0}''(\ell_0) \partial_{\xi}^2 + 1 - 6 \gamma |A|^2, \\ 
		\mathcal{L}_- &= -\frac{1}{2} \omega_{m_0}''(\ell_0) \partial_{\xi}^2 + 1 - 2 \gamma |A|^2,
	\end{align*}
	with the kernels ${\rm Ker}(\mathcal{L}_+) = {\rm span}(A')$ and ${\rm Ker}(\mathcal{L}_-) = {\rm span}(A)$ in $H^2(\mathbb{R})$. Both eigenvectors of the kernel do not belong to $X_R$, hence restricting the space $H^2(\mathbb{R})$ by the two symmetries in (\ref{reversion-scalar}) 
	leads to an invertible linearized operator. The implicit function theorem implies that there is a unique solution 	$A_{\varepsilon} \in H^2(\mathbb{R})$ of the normal form equation (\ref{normal-form-perturbed}) 
	satisfying the reversion symmetry (\ref{reversion-A}). The bound 
$\| A_{\varepsilon} - A \|_{L^{\infty}} \leq C \varepsilon$ for some $C > 0$ follows from the bound (\ref{bound-on-v}) on the remainder terms $v$ and $\tilde{v}$ in (\ref{normal-form-perturbed}). Furthermore, bootstrapping of solutions of equations (\ref{dynamics}) and (\ref{projection}) with polynomial vector fields 
yields smooth solutions $v \in C^{\infty}(\mathbb{R},H^2_{\rm NK}(\Gamma_0))$ and $A_{\varepsilon} \in H^{\infty}(\mathbb{R})$. Since $A_{\varepsilon}$ is a smooth homoclinic orbit of the perturbed equation (\ref{normal-form-perturbed}), then the bound (\ref{C-0-extended}) is satisfied. 
\end{proof}

\section{Numerical simulations}
\label{sec-numerics}

In this section, we present numerical simulations which illustrate the phenomenon of traveling waves on the periodic necklace graph. The simulation have been performed with the help of Grafidi, a Python library devoted to the numerical simulation on quantum graphs with finite differences. The library and its use are described in \cite{BeDuLe22}. For simplicity, we ran the simulations in the case $L_1=L_2=\pi$, where some parameter values are explicitely available and do not need to be numerically approximated. We consider $(\sigma_0,c_0)$ at the bifurcation point (\ref{bif-point-again}) for the corresponding $\ell_0$. Given $c_0$, we find $\ell_0$ from the equation $\omega'(\ell_0)=c_0$ (which is easy if $\omega'$ is known explicitly). For the case $c_0=0.5$, we find $\ell_0 \simeq 0.3437067837$ and $\sigma_0 \simeq 0.0742138336$. From this value we obtain by direct calculations
\[
\omega_{m_0}''(\ell_0)=0.2229344892. 
\]

To ensure that the amplitude profile is scattered among a sufficiently large number of cells, we choose $ \varepsilon $ sufficiently small, e.g. $ \varepsilon  = 0.01$. The necklace is created using the Grafidi library. The number of cells has to be finite and we need to decide how many cells we include in the necklace, and what type of boundary conditions are used at the border of the necklace. For the later, we choose to use Dirichlet boundary conditions, and arrange the number of cells to be sufficiently large so that this choice does not impact the computed behavior. To fix the number of cells, we use the following rule. We first fix a number of cell $n_0$ which is such that, when our initial data is centered on the cell $\Gamma_{n_0}$, its value on the cell $\Gamma_{0}$ is close to machine precision (i.e. $10^{-16}$ in pratice). The envelope is constructed with a solution $A$ to the truncated normal form equation \eqref{normal-form}. Hence $A$ satisfies
\[
|A(\varepsilon x)|\leq C e^{-\varepsilon \kappa |x|}, \quad \kappa := \sqrt{\frac{2}{\omega''(\ell_0)}},
\]
for some $C > 0$. Note that we are at initial time $t=0$, therefore $\xi=x-ct=x$.
The main order in the expansion \eqref{formal-expansion} is $ \varepsilon  u_1$, where $u_1$ is given by \eqref{first-term}. Based on the formula for $u_1$, to be close to machine precision corresponds to having
\[
 \varepsilon  e^{- \varepsilon \kappa |x|}\leq 10^{-16},
\]
i.e., to have $|x|$ such that
\[
|x|\geq -\frac{1}{\varepsilon \kappa}\ln\left(\frac{10^{-16}}{\varepsilon}\right).
\]
Since each cell has a length of $2\pi$, we chose $n_0$ such that
\[
n_0\geq -\frac{1}{2 \pi \varepsilon  \kappa}\ln\left(\frac{10^{-16}}{\varepsilon}\right).
\]
With the values previously chosen for $\ell_0$ and $\varepsilon$, we can take
\[
n_0 = 175. 
\]
In addition, we have to take into account the fact that the solution will be traveling along the necklace. To this aim, we introduce a number of cells $n_1$ corresponding to the number of cells covered by the solution during the simulation time $T$. Since the speed of the wave solution is given by $c_0$ and the length of a cell is $2\pi$, we choose
\[
n_1\geq \frac{c_0 T}{2\pi}.
\]
Finally, the total number of cells in the necklace is chosen to be
\[
2n_0+n_1. 
\]
The Grafidi library is then used to construct the graph. Due to the
large number of cells, we keep the average number of points per
interval relatively small and chose to work with an average of $30$ points per edge. 
The initial data for the evolution is chosen as in \eqref{first-term} (with $t=0$, for which $\xi=x-ct=x$).

To simulate the nonlinear Schr\"odinger evolution on the necklace graph, we use a classical Strang splitting scheme. The idea is to split the equation in two parts, the linear part, which is
\[
i\partial_t\psi+\partial_x^2\psi = 0,
\]
and the nonlinear part
\[
i\partial_t\psi+2|\psi|^2\psi = 0.
  \]
The advantage of the splitting scheme is that the nonlinear part can be solved explicitly. 
Given a time step $\delta_t$, the splitting scheme is  given by
  \[
    \begin{cases}
      \psi^{n+\frac13}=\exp\left({i\delta_t|\psi^n|^2}\right)\psi^n,\\
    \psi^{n+\frac23} = \psi^{n+1/3} - \delta_t \partial_x^2\left(\frac{\psi^{n+\frac23}+\psi^{n+\frac13}}{2}\right),\\
            \psi^{n+1}=\exp\left({i\delta_t\left|\psi^{n+\frac23}\right|^2}\right)\psi^{n+\frac23}.
    \end{cases}
  \]
  In practice, we chose the time step $\delta_t=0.01$. The implementation in the framework of graphs with the Grafidi library is described in more details in \cite{BeDuLe22}. To keep the computation time reasonable, we performed the simulation on a time $T=1/\varepsilon=100$. 
  
The outcome of the simulation can be observed in Figure \ref{fig:necklace_evol}, where we see the pulse propagating along the necklace. The center of mass and the distance between the computed solution and the leading-order approximation $\varepsilon u_1$ in \eqref{first-term} are represented in Figure \ref{fig:necklace_th_exp}. 

\begin{figure}[htpb!]
  \centering
  \includegraphics[width=0.32\textwidth]{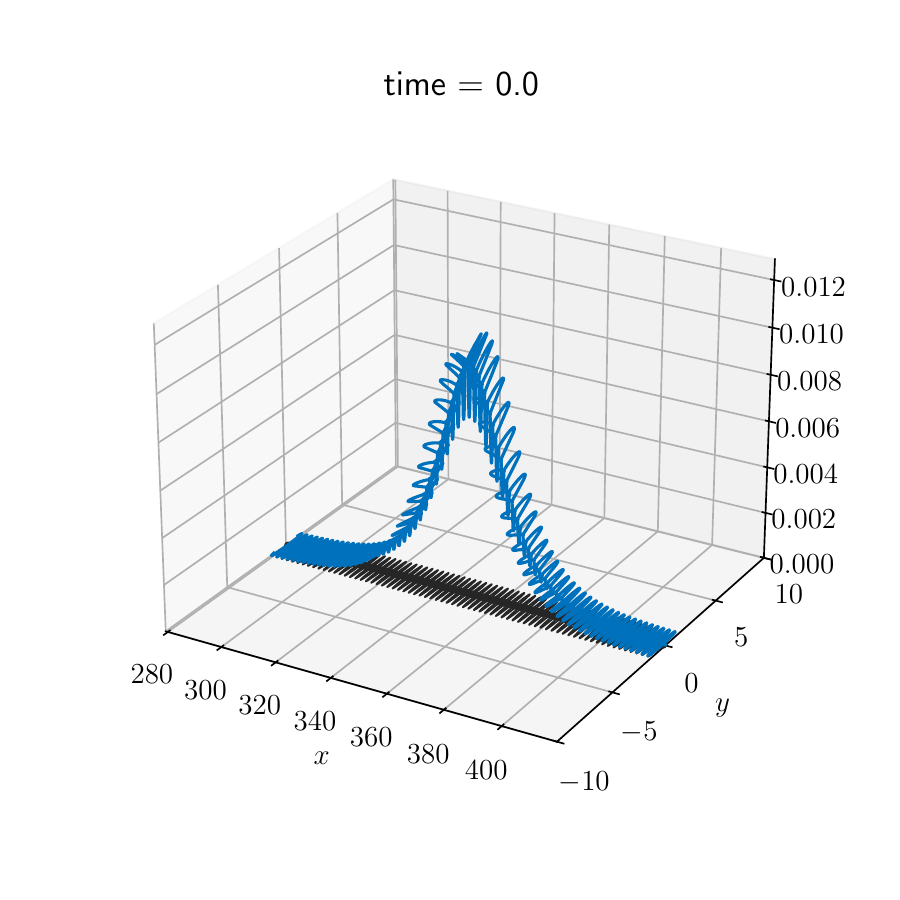}
  \includegraphics[width=0.32\textwidth]{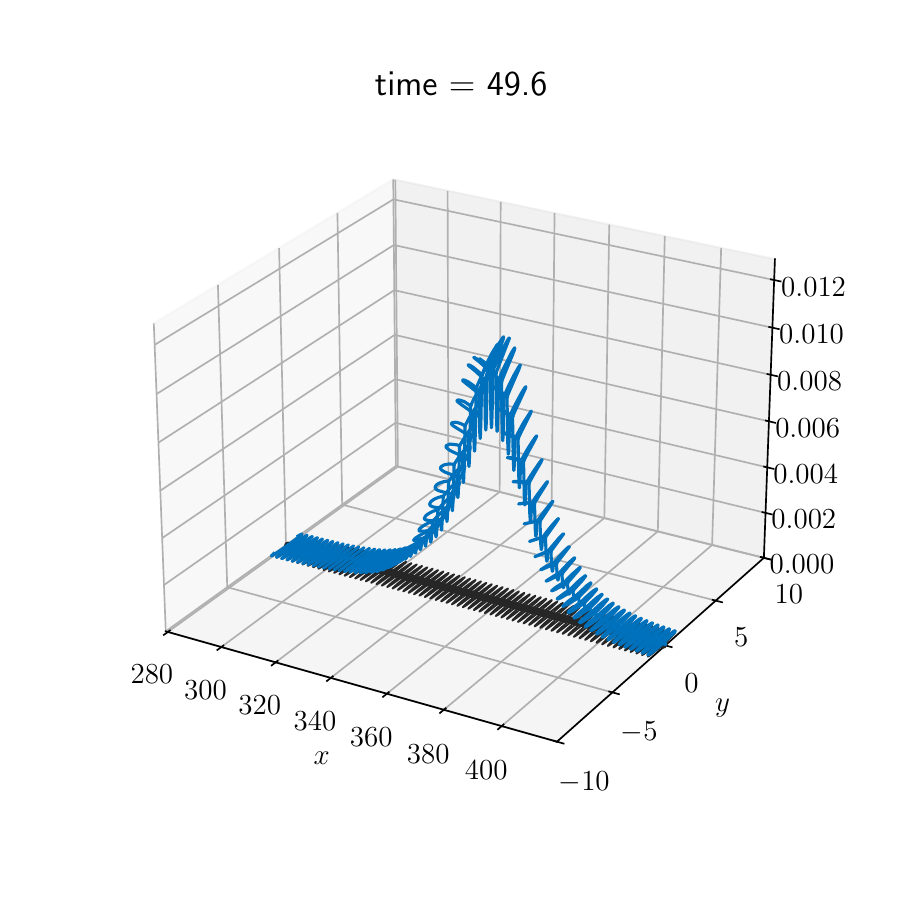}
  \includegraphics[width=0.32\textwidth]{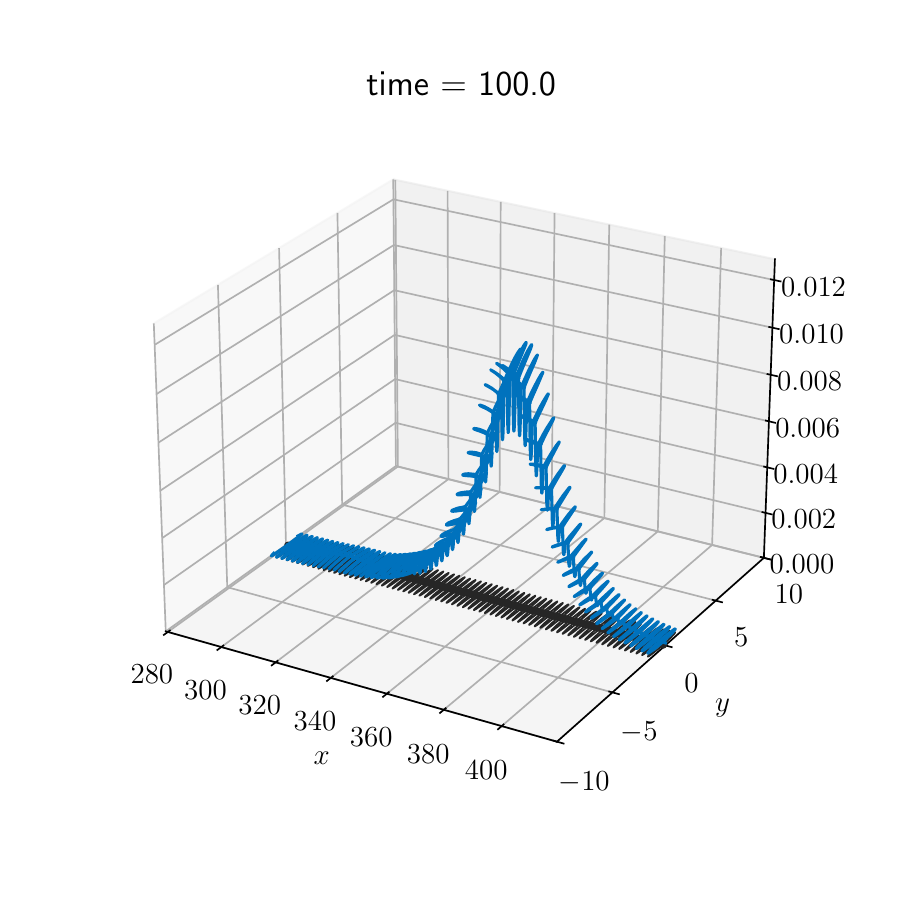}
\caption{Snapshots of the traveling wave along the necklace at initial, middle and final simulation times}
  \label{fig:necklace_evol}
\end{figure}

\begin{figure}[htpb!]
  \centering
\includegraphics[width=0.49\textwidth]{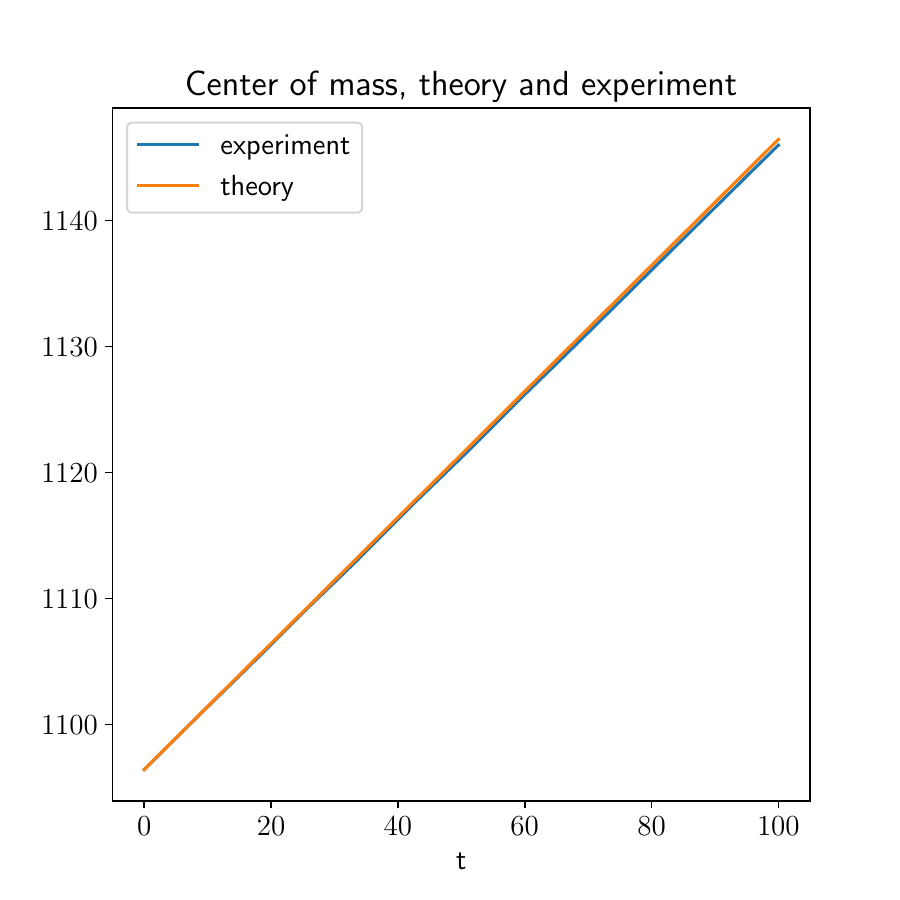}
\includegraphics[width=0.49\textwidth]{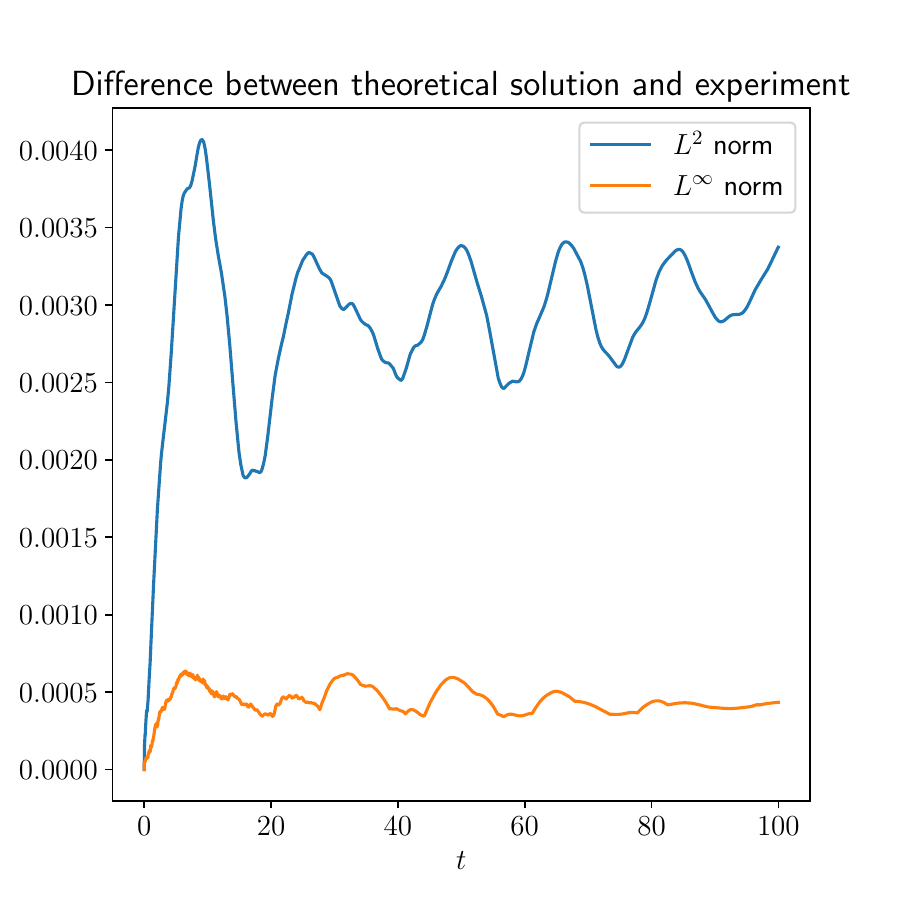}
\caption{Left: plot of the center of mass, in comparison between theory and experiment. 
Right: plot of the $L^2$ and $L^\infty$ norms of the difference between the numerical 
solution and the leading-order approximation $\varepsilon u_1$ in \eqref{first-term}.}
  \label{fig:necklace_th_exp}
\end{figure}

\section{The ladder graph}
\label{app-C}

Let $\Gamma$ be the graph in the shape of a ladder shown in 
Figure \ref{fig:ladder}, where the length of the rungs is denoted by
$L_r$ and the length of a side rail section between two rungs is denoted by
$L_s$. The periodicity cell $\Gamma_0$ of the ladder graph $\Gamma$ can
be considered to be formed by two side rail sections
$\Gamma_{0,s,\pm}$ connected by a rung $\Gamma_{0,r}$, with similar notations on the other cells $(\Gamma_n)_{n\in\mathbb Z}$. Given $n\in\mathbb Z$, the two rail
sections on $\Gamma_n$ are identified with the interval $[nL_s,(n+1)L_s]$ while the rung
is identified with $[-L_r/2,L_r/2]$. 

The set of edges formed by $(\Gamma_{n,s,\pm})_{n\in\mathbb Z}$ and $\Gamma_{n,r}$ can be denoted by $\Gamma_{s,\pm}$ and $\Gamma_r$. 
Both $\Gamma_{s,+}$ and $\Gamma_{s,-}$ are isometric to the real line $\R$. The parts of a function $\psi(x) : \Gamma \to \mathbb C$ can be denoted with similar notations, e.g. $\psi_{s,\pm}$ denotes the part of the function 
$\psi$ on $\Gamma_{s,\pm}$ and $\psi_{r}$ denotes the part of the function 
$\psi$ on $\Gamma_{r}$. 

Bloch eigenfunctions can be introduced by using (\ref{Bloch-wave}) from solutions of the 
spectral problem (\ref{spectral-prob}) with the corresponding Neumann--Kirchhoff boundary conditions 
at the vertex points. The spectral bands are represented in Figure
\ref{fig:spectral-band-ladder_Ls_1_Lr_1} for $L_s = L_r =  1$. 
The blue lines correspond to the bands with the symmetric eigenfunctions satisfying $w_{s,+} =  w_{s,-}$ and the orange lines correspond to the bands with the anti-symmetric eigenfunctions satisfying $w_{s,+} = -w_{s,-}$. The green lines denote the flat bands which only exist for a rational ratio of $L_s$ and $L_r$. 

For the case $L_s = L_r = 1$, expressions for the spectral bands are explicit. The spectral bands for the symmetric eigenfunctions are given by 
\[
\omega(\ell)=\left(\pm\arccos\left(2\cos(\ell)-1\right)+2k\pi\right)^2, \quad k\in\mathbb Z, \quad \ell \in \left[-\frac{\pi}{2},\frac{\pi}{2}\right]. 
\]
The spectral bands for the anti-symmetric eigenfunctions are given by 
\[
\omega(\ell)=\left(\arccos\left(1+2\cos(\ell)\right)+k\pi\right)^2, \quad k\in\mathbb Z, \quad 
\ell \in \left[ -\pi,-\frac{\pi}{2} \right] \cup \left[ \frac{\pi}{2},\pi \right],
\]
augmented with the lowest band at 
\[
\omega(\ell)=-\left(\arccosh\left(1+2\cos(\ell)\right)\right)^2, \qquad \ell \in \left[-\frac{\pi}{2},\frac{\pi}{2}\right]. 
\]
The flat bands are located at $\omega = (2k+1)^2 \pi^2$ for integers $k \geq 0$. 

\begin{figure}[htb!]
	\centering
	\includegraphics[width=0.49\textwidth]{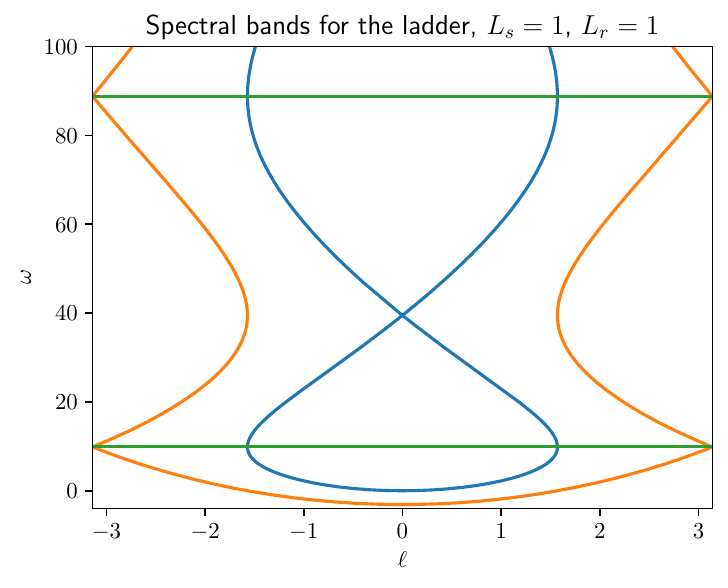}
	\caption{Spectral bands for the ladder graph. The blue lines are for the bands for the symmetric eigenfunctions, the orange lines for the anti-symmetric eigenfunctions, and the green lines for the flat bands.}
	\label{fig:spectral-band-ladder_Ls_1_Lr_1}
\end{figure}

A bifurcation for $\omega = 0$ and $\ell = 0$ is very special for the ladder graph, compared to the necklace graph. The corresponding eigenfunction (\ref{Bloch-wave}) with $\ell_0 = 0$ is constant on $\Gamma$. 
However, if $\psi_{s,+} = \psi_{s,-}$ and $\partial_x \psi_{r} = 0$, then the NLS equation (\ref{NLS}) on the ladder graph $\Gamma$ 
folds into two identical copies of the scalar NLS equation on the real line $\R$ for $\psi_{s,+}$ and $\psi_{s,-}$ since $\Gamma_{s,+}$ and $\Gamma_{s,-}$ are isometric to the real line $\R$. As a result, not only the 
standing wave solution is available explicitly but 
also it continues as the traveling wave solution given by 
\begin{equation}
\label{soliton-on-ladder}
\psi_{s,+}(t,x) = \psi_{s,-}(t,x) = 
e^{i \sigma t} \phi\left(x-ct\right),\qquad \phi(x) = e^{i\frac c2x} \sqrt{\sigma - \frac{c^2}{4}}\sech\left(\sqrt{\sigma - \frac{c^2}{4}} x\right),
\end{equation}
where $\sigma > \frac{c^2}{4}$ and $c \in \mathbb{R}$ are arbitrary parameters. 
The traveling wave is symmetrically placed on the two side rails of $\Gamma$, whereas the value on the rungs of $\Gamma$ are constants given by the values at their vertices:
\begin{equation}
\label{soliton-on-rungs}
\psi_{n,r}(t,x)=e^{i \sigma t} \phi\left((n+1)L_s-ct\right).
\end{equation}
The function $\psi$ in (\ref{soliton-on-ladder}) and (\ref{soliton-on-rungs})  is clearly a traveling wave solution of the NLS equation \eqref{NLS} on the ladder graph $\Gamma$ propagating at speed $c$ along the graph $\Gamma$. The spatial dynamical system (\ref{NLS-dynamics}) for the traveling waves (\ref{trav-solution}) also has an infinite-dimensional center manifold due to resonances with other spectral bands in Figure \ref{fig:spectral-band-ladder_Ls_1_Lr_1}. However, there exists an invariant reduction to the two-dimensional subspace isolated from the stable, unstable, and the rest of the center manifold. This special situation is very similar to the homogeneous limit of the necklace graph in Appendix \ref{app-A}. 

For any other bifurcation on the ladder graph with $\omega \neq 0$ and $\ell \neq 0$, the eigenfunction (\ref{Bloch-wave}) is no longer constant on $\Gamma_r$. In this case, there is no reduction of the NLS equation (\ref{NLS}) on the ladder graph $\Gamma$ to the NLS equation on the real line $\R$. Bifurcating standing wave solutions are not available explicitly and bifurcating traveling wave solutions generally have small oscillatory tails due to the infinite-dimensional center manifold of the spatial dynamical system. 

\appendix 

\section{Traveling waves in the homogeneous case}
\label{app-A}

For the necklace graph in the homogeneous case with $L_1 = 2\pi$ and $L_2 = 0$, the boundary conditions (\ref{vertex-5}) and (\ref{vertex-6}) are simply the periodic boundary conditions for $\tilde{\phi} \equiv \tilde{\phi}_0$ and $\tilde{\eta} \equiv \tilde{\eta}_0$. The phase space $\mathcal{D}$ of the spatial dynamical system (\ref{NLS-dynamics}) is then given by 
\begin{equation}
\label{D-periodic}
\mathcal{D} := \left\{ (\tilde{\phi},\tilde{\eta}) : \;\; \tilde{\phi} \in H^2_{\rm per}, \;\; \tilde{\eta} \in H^1_{\rm per} \right\}.
\end{equation}
In this case, traveling waves in the spatial dynamical system (\ref{NLS-dynamics}) coincide with traveling waves of the focusing NLS equation on the line $\mathbb{R}$, according to the following proposition. 

\begin{proposition}
    \label{prop-line}
    The spatial dynamical system (\ref{NLS-dynamics}) with $L_1 = 2\pi$ and $L_2 = 0$ admits a smooth homoclinic orbit for every $\sigma > \frac{c^2}{4}$, in fact, in the explicit form 
    \begin{equation}
\label{NLS-soliton}
\tilde{\phi}(\xi) = e^{\frac{ic \xi}{2}} \sqrt{\sigma - \frac{c^2}{4}} {\rm sech}\left(\sqrt{\sigma - \frac{c^2}{4}} \xi \right), \quad \tilde{\eta}(\xi) = \tilde{\phi}'(\xi). 
\end{equation}
The translational parameters of the homoclinic orbits are uniquely chosen to satisfy the reversibility constraints: ${\rm Im}(\tilde{\phi})(0) = 0$ and ${\rm Re}(\tilde{\eta})(0) = 0$.
\end{proposition}

\begin{proof}
We use the Fourier basis in $\mathcal{D}$, given by (\ref{D-periodic}),
and decompose the solution in the form:
\begin{equation}
\label{decomposition}
\tilde{\phi}(\xi,x) = \sum_{k \in \Z} u_k(\xi) e^{i k x}, \quad 
\tilde{\eta}(\xi,x) = \sum_{k \in \Z} r_k(\xi) e^{i k x}.
\end{equation}
The spatial dynamical system (\ref{NLS-dynamics}) is now rewritten in the form 
\begin{equation}
\label{Fourier-modes-k}
\frac{d \vec{w}_k}{d \xi}  = A_k \vec{w}_k + \left( \begin{matrix} 0 \\ N_k \end{matrix} \right),
\end{equation}
where 
$$
\vec{w}_k := \begin{pmatrix} u_k  \\ r_k 
\end{pmatrix}, \qquad A_k := \begin{pmatrix} 0  & 1  \\ 
k^2 + \sigma & -2ik +ic \end{pmatrix}
$$
and 
$$
N_k := -2\sum\limits_{k_1 \in \Z} u_{k_1} u_{k_2} \bar{u}_{k_1+k_2-k}.
$$
Since 
\begin{equation}
    \label{reduction-X-0}
N_k |_{X_0} = 0, \quad k \in \Z\backslash \{0\}, \quad  \mbox{\rm at} \;\; X_0 := \left\{ 
\{ \vec{w}_k \}_{k \in \Z} : \quad \vec{w}_k = \vec{0}, \quad k \in \Z\backslash \{0\} \right\},
\end{equation}
the system (\ref{Fourier-modes-k}) admits the invariant reduction in $X_0$, for which it reduces to the second-order system:
\begin{equation}
\label{Fourier-modes-0}
\frac{d}{d \xi} \begin{pmatrix} u_0 \\ r_0 
\end{pmatrix} = \begin{pmatrix} 0 & 1 \\
\sigma & ic \end{pmatrix} \begin{pmatrix} u_0 \\ r_0
\end{pmatrix} - \begin{pmatrix} 0 \\ 
2|u_0|^2 u_0 \end{pmatrix},
\end{equation}
which is equivalent to the scalar  second-order equation
\begin{equation}
\label{NLS-trav}
u_0''(\xi) -i c u_0'(\xi) - \sigma u_0(\xi) + 2 |u_0(\xi)|^2 u_0(\xi) = 0.
\end{equation}
This is the stationary NLS equation on the real line with the exact solution (\ref{NLS-soliton}) for the homoclinic orbit, uniquely selected by the reversibility constraints: ${\rm Im}(\tilde{\phi})(0) = 0$ and ${\rm Re}(\tilde{\eta})(0) = 0$. The homoclinic orbit exists for $\sigma > \frac{c^2}{4}$.
\end{proof}

\begin{remark}
     The homoclinic orbit of Proposition \ref{prop-line} bifurcates at $\sigma = \frac{c^2}{4}$ for $\sigma > \frac{c^2}{4}$, in agreement with Example \ref{example-expansion}. In addition to the bifurcating eigenvalues (blue dots in Figure \ref{fig:spectral-picture-L1_0}), there exist infinitely many eigenvalues on the imaginary axis (see Figure \ref{fig:spectral-picture-L1_0}) which generally result in the small oscillatory tails. However, due to the existence of the invariant reduction on $X_0$, see (\ref{reduction-X-0}), there exists a true solitary wave solution with zero oscillatory tails given by Proposition \ref{prop-expansion}. The invariant reduction on $X_0$ is broken on the necklace graph $\Gamma$ with $L_1 \neq 2\pi$ and $L_2 \neq 0$, for which the small nonzero oscillatory tails generally exist. 
\end{remark}

\section{Variational formulation}
\label{app-B}

Minimization techniques are a classical tool to obtain the existence  of standing wave profiles of the NLS equations  (see e.g. \cite{BeLi83-1}). On the contrary, as was mentioned in Section \ref{sec-2}, existence of solutions of the spatial NLS equation \eqref{NLS-potential} is not easily obtained via minimization. Indeed, it turns out that the classical minimization problems used to get a critical point of the action functional $\Lambda_{\sigma,c}$, defined in \eqref{var-form}, degenerate in the present setting, even when $c=0$. We illustrate this fact with the following result, in which we consider minimization of the quadratic part of $\Lambda_{\sigma,0}$ over a potential part. We have generalized  here the cubic nonlinearity to a super-linear power nonlinearity. Minimization fails due to the lack of control of the quadratic part  over the potential norm. In other words, in comparison to the classical case, what is missing here is a Gagliardo-Nirenberg type inequality.

  \begin{proposition}
    \label{prop:min-prob}
    Let $p>1$ and $\sigma>0$. Define the space $\mathcal H$ by 
  \[
    \mathcal H=\left\{
      \varphi\in H^1(\mathbb R\times \Gamma_0)\cap C^0(\mathbb R\times \bar\Gamma_0)
      :\varphi_0(\xi,L_1)=\varphi_\pm(\xi,L_1),\,\varphi_0(\xi,0)=\varphi_\pm(\xi,2\pi)
    \right\}.
  \]
    The infimum of the constrained variational problem 
    \begin{equation}
      \label{eq:min-prob}
      \inf_{\varphi \in \mathcal{H}} \left\{\int_{\mathbb{R} \times \Gamma_0} \left(  |(\partial_{\xi} + \partial_x) \varphi|^2 + \sigma |\varphi|^2 
        \right) d\xi dx : \quad  \int_{\mathbb{R} \times \Gamma_0}|\varphi|^{p+1} d\xi dx=1\right\}     
    \end{equation}
    is $0$ and it is not achieved. 
  \end{proposition}

  \begin{proof}
    To show that the minimization problem \eqref{eq:min-prob} does not admit a solution, we construct a sequence of functions $(\varphi_n)_{n \in \mathbb{N}} \subset \mathcal H$ such that
    \[
      \int_{\mathbb{R} \times \Gamma_0}|\varphi_n|^{p+1} d\xi dx=1,\quad \int_{\mathbb{R} \times \Gamma_0} \left(  |(\partial_{\xi} + \partial_x) \varphi_n|^2 + \sigma |\varphi_n|^2 
        \right) d\xi dx\to 0\text{ as }n\to \infty,
      \]
      and $(\varphi_n)_{n \in \mathbb{N}}$ does not admit a limit in $\mathcal H$. We will use a separation of variable technique. Let $f\in H^1(\R)$ and $g\in H_{\rm C}^1(\Gamma_0)$. The function $g$ on $\Gamma_0$ is identified with its three components $(g_0,g_+,g_-)$ defined on $[0,2\pi]$. We extend  $(g_0,g_+,g_-)$ to $\mathbb R$ by periodicity and denote the corresponding function on $\Gamma$ also by $g$. We define $\varphi\in\mathcal H$ by
      \[
        \varphi(\xi,x)=f(\xi)g(\xi-x).
      \]
      By construction, $\varphi$ satisfies
      \begin{align*}
                \int_{\mathbb{R} \times \Gamma_0}|\varphi|^{p+1} d\xi dx&=
\left(\int_{\mathbb{R}}|f|^{p+1} d\xi 
\right)
\left(\int_{ \Gamma_0}|g|^{p+1} dx
\right)
        ,\\
        \int_{\mathbb{R} \times \Gamma_0}|\varphi|^{2} d\xi dx&=
\left(\int_{\mathbb{R}}|f|^{2} d\xi 
\right)
\left(\int_{ \Gamma_0}|g|^{2} dx
\right).
      \end{align*}
      Moreover, we have
      \[
        (\partial_\xi+\partial_x)\varphi(\xi,x)=f'(\xi)\cdot g(\xi-x)+f(\xi)\cdot g'(\xi-x)-f(\xi)\cdot g'(\xi-x)=f'(\xi)\cdot g(\xi-x).
      \]
      Therefore
      \[
\int_{\mathbb{R} \times \Gamma_0} |(\partial_{\xi} + \partial_x) \varphi|^2 
d\xi dx=
\int_{\mathbb{R} \times \Gamma_0}
|f'(\xi)\cdot g(\xi-x)|^2d\xi dx =
\left(
\int_{\mathbb{R}}|f'|^2d\xi
\right)
\left(
\int_{\Gamma_0}|g|^2dx
\right).
\]
Let $g\in H^1_{\rm C}(\Gamma_0)$ be fixed and define $\mu,\nu>0$ by
\[
\norm{g}_{L^{p+1}(\Gamma_0)}=\mu,\quad  \norm{g}_{L^{2}(\Gamma_0)}=\nu. 
\]
Consider the partial  minimization problem
  \begin{multline*}
    \min_{f \in H^1(\mathbb{R})} \Big\{\int_{\mathbb{R} \times \Gamma_0} \left(  |(\partial_{\xi} + \partial_x) \varphi|^2 + \sigma |\varphi|^2 
        \right) d\xi dx
        :\\
        \varphi(\xi,x)=f(\xi)\cdot g(\xi-x), \;\; \int_{\mathbb{R} \times \Gamma_0}|\varphi|^{p+1} d\xi dx=1\Big\}.
      \end{multline*}
    As $g$ is fixed, using the notation previously introduced, this problem might be rewritten as
    \[
    \min_{f \in H^1(\mathbb{R})} \left\{\nu^2\left(\int_{\mathbb{R} }  |f'|^2 + \sigma\int_{\mathbb{R} }|f|^2d\xi\right) : \quad \int_{\mathbb{R} }|f|^{p+1} d\xi=\mu^{-(p+1)}\right\}.
    \]
    This problem  is well-known \cite{BeLi83-1} to admit for any $1<p<\infty$ a unique positive even minimizer, which is constructed with the classical ground state $Q$, i.e.
\[
Q(\xi )=\left(\frac{(p+1)\sqrt{\sigma}}{2}\right)^{\frac{1}{p -1}} \sech \left(\frac{\left(p -1\right)\sqrt{\sigma} \xi }{2}\right)^{\frac{2}{p -1}},
\]
and is given by
\[
  f_{\mu^{-1}}=\frac{\mu^{-1} Q}{\norm{Q}_{L^{p+1}(\mathbb R)}}.
\]
The value of the minimum is then given by
\begin{align*}
  \nu^2\left(\int_{\mathbb{R} }  |f_{\mu^{-1}}'|^2 + \sigma\int_{\mathbb{R} }|f_{\mu^{-1}}|^2d\xi\right) &=
  \left(\frac{\nu}{\mu}\right)^2\left(
\norm{Q'}_{L^{2}(\mathbb R)}^2+\sigma \norm{Q}_{L^{2}(\mathbb R)}^2
\right)\\
&=
 \left(\frac{\norm{g}_{L^2(\Gamma_0)}}{\norm{g}_{L^{p+1}(\Gamma_0)}}\right)^2
\norm{Q}_{L^{p+1}(\mathbb R)}^{p+1},
\end{align*}
where we have used $\norm{Q'}_{L^{2}(\mathbb R)}^2+\sigma \norm{Q}_{L^{2}(\mathbb R)}^2 = \norm{Q}_{L^{p+1}(\mathbb R)}^{p+1}$ for the ground state $Q$ in the last equality. 
The construction of a minimizing sequence for the original problem is reduced by using the separation of variables to the search for a sequence of functions $(g_n)_{n \in \mathbb{N}} \subset H^1_{\rm C}(\Gamma_0)$ such that
\[
\lim_{n\to\infty}\frac{\norm{g_n}_{L^2(\Gamma_0)}}{\norm{g_n}_{L^{p+1}(\Gamma_0)}}=0.
  \]
Such  a sequence can easily be constructed, taking e.g. a function $g:\Gamma\to\mathbb R$ compactly supported in $[0,L_1]$ and considering $g_n(x)=n^{\frac{1}{p+1}}g(nx)$, $n \in \mathbb{N}$. This concludes  the proof. 
    \end{proof}

\bibliographystyle{siam} %
\bibliography{biblio}

\end{document}